\documentclass[english]{scrartcl}
\usepackage{fullpage}

\usepackage{setspace}
\usepackage[hang]{footmisc}
\usepackage{amsthm}
\usepackage{amsmath}
\usepackage{amssymb}
\usepackage{mathtools}
\usepackage{graphicx}
\usepackage{pgf,tikz}
\usepackage[pdftex,colorlinks,backref=page,citecolor=blue,bookmarks=false]{hyperref}
\usepackage{thm-restate}
\usepackage{enumitem}
\usepackage{xcolor}
\usepackage{subcaption}
\usepackage{algorithm}
\usepackage{algpseudocode}
\usepackage{pifont}
\usepackage{multirow}

\usepackage{tabularx}
\usepackage{booktabs}

\usepackage{stmaryrd}

\usepackage{caption}
\usepackage[square,sort,comma,numbers]{natbib} 
\usepackage{hyperref}
\usepackage{cleveref}

\theoremstyle{plain}

\newtheorem*{theorem*}{Theorem}
\newtheorem{theorem}{Theorem}[section]
\crefname{theorem}{Theorem}{Theorems}
\Crefname{theorem}{Theorem}{Theorems}

\newtheorem*{lemma*}{Lemma}
\newtheorem{lemma}[theorem]{Lemma}
\crefname{lemma}{Lemma}{Lemmas}
\Crefname{lemma}{Lemma}{Lemmas}

\newtheorem*{claim*}{Claim}
\newtheorem{claim}[theorem]{Claim}
\crefname{claim}{Claim}{Claims}
\Crefname{claim}{Claim}{Claims}

\newtheorem{proposition}[theorem]{Proposition}
\crefname{proposition}{Proposition}{Propositions}
\Crefname{proposition}{Proposition}{Propositions}

\newtheorem{corollary}[theorem]{Corollary}
\crefname{corollary}{Corollary}{Corollaries}
\Crefname{corollary}{Corollary}{Corollaries}

\crefname{conjecture}{Conjecture}{Conjectures}
\Crefname{conjecture}{Conjecture}{Conjectures}

\newtheorem{question}[theorem]{Question}
\crefname{question}{Question}{Questions}
\Crefname{question}{Question}{Questions}

\crefname{observation}{Observation}{Observations}
\Crefname{observation}{Observation}{Observations}

\crefname{example}{Example}{Examples}
\Crefname{example}{Example}{Examples}

\theoremstyle{definition}

\crefname{problem}{Problem}{Problems}
\Crefname{problem}{Problem}{Problems}

\newtheorem{definition}[theorem]{Definition}
\crefname{definition}{Definition}{Definitions}
\Crefname{definition}{Definition}{Definitions}

\newtheorem{fact}[theorem]{Fact}
\crefname{fact}{Fact}{Facts}
\Crefname{fact}{Fact}{Facts}

\theoremstyle{remark}

\crefname{remark}{Remark}{Remarks}
\Crefname{remark}{Remark}{Remarks}

\usepackage{xpatch}
\xpatchcmd{\proof}\emph{shape}{\normalfont\proofnamefont}{}{}
\newcommand{\proofnamefont}{}
\renewcommand{\proofnamefont}{\bfseries}

\newcommand{\remove}[1]{}

\makeatletter
\newcommand{\xnrightarrow}[2][]{%
  \mathrel{%
    \vphantom{\xrightarrow[#1]{#2}}%
    \ooalign{\hidewidth\neg@arrow\hidewidth\cr$\m@th\xrightarrow[#1]{#2}$\cr}%
  }%
}
\newcommand{\neg@arrow}{%
  $\m@th\vcenter{\hbox{%
    \rotatebox[origin=c]{-45}{\scalebox{1.5}[1]{$\m@th\scriptscriptstyle|$}}%
  }}$
}
\makeatother

\usepackage{centernot}
\newcommand{\constrained}{\xrightarrow{\mathrm{c-ram}}}
\newcommand{\notconstrained}{\xnrightarrow{\mathrm{c-ram}}}
\newcommand{\cram}{\mathrm{c-ram}}
\newcommand{\aram}{\mathrm{a-ram}}

\newcommand{\antiramsey}{\xrightarrow{\mathrm{a-ram}}}
\newcommand{\notantiramsey}{\xnrightarrow{\mathrm{a-ram}}}

\newcommand{\vecT}{\bar{T}}

\renewcommand{\Pr}{\mathbb{P}}
\newcommand{\cP}{\mathcal{P}}
\newcommand{\cF}{\mathcal{F}}
\newcommand{\eps}{\varepsilon}

\title{Thresholds for constrained Ramsey and anti-Ramsey problems}

\author{
    Natalie Behague\thanks{Mathematics Institute, University of Warwick, Coventry CV4 7AL, UK. \texttt{natalie.behague@warwick.ac.uk}. This research was supported by a PIMS postdoctoral fellowship while at the University of Victoria. 
    }    
    \and
    Robert Hancock\thanks{
    Mathematical Institute, University of Oxford, Andrew Wiles Building, Radcliffe Observatory Quarter,  Woodstock Rd, Oxford OX2~6GG, UK.  \texttt{robert.hancock@maths.ox.ac.uk}. Research supported by ERC Advanced Grant 883810 and by a Humboldt Research Fellowship at Heidelberg University.
    }
    \and
    Joseph Hyde\thanks{Department of Mathematics, King's College London, Strand Building, Strand Campus, Strand, London, WC2R~2LS, UK. \texttt{josephhyde@uvic.ca}. 
    } 
    \and
    Shoham Letzter\thanks{
    Department of Mathematics, 
    University College London, 
    Gower Street, London WC1E~6BT, UK. 
     \texttt{s.letzter}@\texttt{ucl.ac.uk}. 
    Research supported by the Royal Society.
    }
    \and
    Natasha Morrison\thanks{Department of Mathematics and Statistics, University of Victoria, 3800 Finnerty Road, Victoria, BC, V8P~5C2, Canada. \texttt{nmorrison@uvic.ca}. 
    Research supported by NSERC Discovery Grant RGPIN-2021-02511 and NSERC Early Career Supplement DGECR-2021-00047.
    } 
}
\definecolor{lessgarishgreen}{rgb}{0,0.6,0}

\begin{document}

\maketitle

\begin{abstract}
Let $H_1$ and $H_2$ be graphs. A graph $G$ has the \emph{constrained Ramsey property for $(H_1,H_2)$}
if every edge-colouring of $G$ contains either a monochromatic copy of $H_1$ or a rainbow copy of $H_2$. 
Our main result gives a 0-statement for the constrained Ramsey property in $G(n,p)$ whenever $H_1 = K_{1,k}$ for some $k \ge 3$ and $H_2$ is not a forest. Along with previous work of Kohayakawa, Konstadinidis and Mota, this resolves the constrained Ramsey property for all non-trivial cases with the exception of $H_1 = K_{1,2}$, which is equivalent to the anti-Ramsey property for $H_2$.

For a fixed graph $H$, we say that $G$ has the \emph{anti-Ramsey property for $H$} if any proper edge-colouring of $G$ contains a rainbow copy of $H$. We show that the 0-statement for the anti-Ramsey problem in $G(n,p)$ can be reduced to a (necessary) colouring statement, and use this to find the threshold for the anti-Ramsey property for some particular families of graphs.
\end{abstract}
\section{Introduction}

For fixed graphs $H_1,H_2$, we say that a graph $G$ has the \emph{constrained Ramsey property for $(H_1,H_2)$}, denoted $G \constrained (H_1,H_2)$,
if any edge-colouring of $G$ contains either a monochromatic copy of $H_1$ or a rainbow copy of $H_2$, i.e.\ a copy of $H_2$ where each edge has a different colour.
It is not hard to see that $G$ can not have the constrained Ramsey property unless either $H_1$ is a star or $H_2$ is a forest. 
Indeed, label $V(G)$ (arbitrarily) as $v_1, \ldots, v_{v(G)}$ and define a colouring $\chi$ such that $\chi(v_i v_j) = \min\{i,j\}$ for every edge $v_iv_j$.
Then the only monochromatic graphs in $\chi$ are stars (as every edge coloured $i$ touches the vertex $v_i$), 
and every cycle in $\chi$ contains at least two edges of the same colour (if $i$ is the smallest index of a vertex in a cycle $C$,
then $C$ has two edges coloured $i$). 
Hence we have $G \notconstrained (H_1,H_2)$ if $H_1$ is not a star and $H_2$ is not a forest. 

In this paper, we are interested in determining when a graph typically has the constrained Ramsey property for a given pair $(H_1,H_2)$. 
This is formalised by asking when the \emph{random graph} $G(n,p)$ (which has $n$ vertices and where each possible edge is included independently with probability $p$) is likely to have the constrained Ramsey property. 

A \emph{graph property} is a collection of graphs, and we say that a graph property $\cP$ is \emph{monotone (increasing)} if,
whenever $H$ and $G$ are graphs satisfying $H \in \cP$, $H \subseteq G$ and $V(G) = V(H)$, then $G \in \cP$.
Throughout this paper, we will say that a function $f: \mathbb{N} \to \mathbb{R}$ is a \emph{coarse threshold function for $\cP$ in $G(n,p)$} if \[\lim\limits_{n \to \infty} \mathbb{P}[G(n,p) \in \cP] =   
\begin{cases}
    0       & \quad \text{if } p = o(f(n)),\\
    1       & \quad \text{if } p = \omega(f(n)),
\end{cases}\] 
and a \emph{semi-sharp threshold function for $\cP$ in $G(n,p)$} 
if there exist constants $c, C > 0$ such that \[\lim\limits_{n \to \infty} \mathbb{P}[G(n,p) \in \cP] =   
\begin{cases}
    0       & \quad \text{if } p \leq  c f(n),\\
    1       & \quad \text{if } p \geq  C f(n).
\end{cases}\]

Bollob\'{a}s and Thomason~\cite{monotone-property} showed that every non-trivial\footnote{A property $\cP$ is \emph{non-trivial} if, for every large enough $n$, there exist $n$-vertex graphs $H$ and $G$ such that $H\in \cP$ and $G\not\in \cP$.}
monotone property $\cP$ for $G(n,p)$ has a coarse threshold. 
We will use the term \emph{0-statement} to refer to the statement for $p = o(f(n))$ or $p \leq cf(n)$ (depending on context), and \emph{1-statement} to refer to the statement for $p = \omega(f(n))$ or $p \geq Cf(n)$. 

Note that a $1$-statement is equivalent to saying that with high probability (that is, with probability tending to $1$ as $n$ tends to infinity, 
which we will henceforth abbreviate to w.h.p.) a graph on $n$ vertices with density much greater than $f(n)$ has the property $\cP$, 
whereas the $0$-statement says that w.h.p.\ a graph on $n$ vertices with density much less than $f(n)$ does not have the property $\cP$.

\begin{question}[Constrained Ramsey]\label{ques:constrainedramsey}
    Let $H_1,H_2$ be graphs such that $H_1$ is a star or $H_2$ is a forest. What is a coarse/semi-sharp threshold function for the constrained Ramsey property for $(H_1, H_2)$ in $G(n,p)$? 
\end{question}

Before discussing what is currently known for various $H_1$ and $H_2$, 
we need to introduce some definitions. For a graph $H$,
we define $d(H) := \frac{e(H)}{v(H)}$ and the \emph{density} of $H$ to be
\[
	m(H) \coloneqq \max\{d(J) : J \subseteq H \text{ and } v(J) \ge 1\}.
\]
We say that $H$ is \emph{balanced} if $m(H) = d(H)$, and \emph{strictly balanced} if for all proper subgraphs $J \subsetneq H$ with $v(J) \ge 1$, we have $d(J) < d(H)$.

Similarly, we define 
\[d_2(H) \coloneqq 
    \begin{cases}
		\frac{e(H) - 1}{v(H) - 2}   & \quad \text{if $e(H) \ge 1$ and $v(H) \geq 3$},\\
        \frac{1}{2}                 & \quad \text{if} \ H \cong K_2,\\    
        0                           & \quad \text{otherwise,}
    \end{cases}\] 
and the \emph{$2$-density of $H$} to be \[m_2(H) \coloneqq \max\{d_2(J): J \subseteq H\}.\] 
    
We say that $H$ is \emph{$2$-balanced} if $m_2(H) = d_2(H)$, and \emph{strictly $2$-balanced} if for all proper subgraphs $J \subsetneq H$, we have $d_2(J) < d_2(H)$.

When $H_1$ is connected and is not a star, the random Ramsey theorem of R\"odl and Ruci\'nski~\cite{rrlower,rr2,rr} states that the threshold for $G(n,p)$ having the Ramsey property for $H_1$ -- that is, for all $r \geq 2$, for any $r$-colouring of the edges of $G(n,p)$ there is a monochromatic copy of $H_1$ -- is $n^{-1/m_2(H_1)}$.
Since when $H_2$ has at least three edges, any $2$-colouring automatically avoids a rainbow copy of $H_2$,
so in this case a $0$-statement for the constrained Ramsey property for $(H_1,H_2)$ holds at $n^{-1/m_2(H_1)}$. 
Collares, Kohayakawa, Moreira and Mota \cite{c-ram_threshold} proved a $1$-statement for all forests $H_2$ with at least three edges,  
and further obtained the location of the threshold for all remaining cases where $H_2$ is a forest.
Note however that the threshold is not always explicit: in some cases it is given in terms of the infimum of $m(F)$ taken over an infinite family of graphs $F$.
See the full collection of known results about thresholds for the constrained Ramsey property for $(H_1, H_2)$ in Table~\ref{tab:constrained} in Section~\ref{sec:conclusion}.

We therefore focus our attention exclusively on the case where $H_1$ is a star $K_{1,k}$ with $k \ge 2$ and $H_2$ is not a forest. 

For any fixed $k \geq 2$, a natural candidate for the threshold function of the constrained Ramsey property for $(K_{1,k}, H_2)$ in $G(n,p)$ is $n^{-1/m_2(H_2)}$. 
Indeed, the expected number of edges of $G(n,p)$ is $\Theta(n^2p)$ and the expected number of copies of $H_2$ in $G(n,p)$ is $\Theta(n^{v(H_2)}p^{e(H_2)})$.
One can observe that if $p = n^{-1/m_2(H_2)}$ and $H_2$ is $2$-balanced then these quantities are of the same order. 
Therefore if $p \leq cn^{-1/m_2(H_2)}$, for a small enough constant $c > 0$, then we have intuitively that every edge belongs to very few (possibly zero) copies of $H_2$.
We can then hope to colour each copy of $H_2$ such that the colourings of different copies barely interact with each other.
In particular, we can hopefully use for each copy $H'$ of $H_2$ a colour $c_{H'}$ twice, such that for any two copies $H'$ and $H''$ of $H_2$, we have $c_{H'} \ne c_{H''}$. This way we avoid a rainbow copy of $H_2$, and it suffices to show that this is possible without creating a monochromatic $K_{1,k}$.

The 1-statement for the constrained Ramsey property in the case where $H_1$ is a star is a result of Kohayakawa, Konstadinidis and Mota \cite{anti-ram_1-statement}. 
Note that in~\cite{anti-ram_1-statement}, the result is instead phrased in terms of the threshold for having the property that any locally $r$-bounded edge-colouring (for every vertex,
no colour appears more than $r$ times on the edges incident to it) produces a monochromatic copy of $H$. This is clearly equivalent to the constrained Ramsey property for $(K_{1,{r+1}},H)$.
\begin{theorem}[{\cite[Theorem 2.2]{anti-ram_1-statement}}]\label{thm:1-statement}
    Let $k \ge 2$ and $H$ be a graph. There exists a constant $C>0$ such that if $p \geq Cn^{-1/m_2(H)}$ then
    \[
    \lim\limits_{n \to \infty} \mathbb{P}[G(n,p) {\constrained} (K_{1,k}, H)] =  1.
    \]
\end{theorem}

In the following two theorems we complete the picture for the constrained Ramsey problem in random graphs, with the exception of the case when $H_1 = K_{1,2}$. 
Note that the $1$-statements of these results follow immediately from Theorem~\ref{thm:1-statement}.

\begin{theorem}\label{thm:constrained}
    Let $k \ge 3$ and $H$ be a graph that is not a forest. If $k \ge 4$, or $m_2(H) \ne 2$, or $m_2(H)=2$ and $H$ contains a strictly $2$-balanced graph $J \ne K_3$ with $m_2(J)=2$, then there exist constants $c, C>0$ such that
    \[\lim\limits_{n \to \infty} \mathbb{P}[G(n,p) {\constrained} (K_{1,k}, H)] =   
\begin{cases}
    0       & \quad \text{if } p \leq cn^{-1/m_2(H)},\\
    1       & \quad \text{if } p \geq Cn^{-1/m_2(H)}.
\end{cases}\]  
\end{theorem}
The cases not covered by Theorem~\ref{thm:constrained}, where $k=3$, $m_2(H)=2$ and the unique strictly $2$-balanced graph $J$ with $m_2(J)=2$ contained in $H$ is $K_3$ (e.g.~when $H$ itself is $K_3$), require a weaker 0-statement.
\begin{theorem}\label{thm:constrained_K3} Let $H$ be a graph where $m_2(H)=2$ and $K_3$ is the unique strictly $2$-balanced graph with $2$-density equal to $2$ contained in $H$.
 There exists a constant $C > 0$ such that
    \[\lim\limits_{n \to \infty} \mathbb{P}[G(n,p) {\constrained} (K_{1,3}, H)] =   
\begin{cases}
    0       & \quad \text{if } p = o\!\left(n^{-1/2}\right),\\
    1       & \quad \text{if } p \geq Cn^{-1/2}.
\end{cases}\]  
Moreover, for any constant $c>0$ there exists $\zeta= \zeta(c) > 0$ such that if $p = cn^{-1/2}$ then $\lim\limits_{n \to \infty}\mathbb{P}[G(n,p) {\constrained} (K_{1,3}, K_3)] \ge \zeta$. 
\end{theorem}
The second part of the statement shows that for $H=K_3$ the threshold is coarse, and it cannot be improved to semi-sharp.

The proofs of \Cref{thm:constrained,thm:constrained_K3} can be found in Section~\ref{sec:constrained}.

For a fixed graph $H$, we say that $G$ has the \emph{anti-Ramsey property for $H$}, written $G \antiramsey H$, if any \emph{proper} edge-colouring of $G$ contains a rainbow copy of $H$.
If $H_1 = K_{1,2}$, the star with two edges, then an edge-colouring of $G$ avoiding a monochromatic copy of $H_1$ is precisely a proper colouring. 
That is, having the constrained Ramsey property for ($K_{1,2}, H_2$) is equivalent to having the anti-Ramsey property for $H_2$. 

The study of the anti-Ramsey problem was initiated by R\"{o}dl and Tuza \cite{anti-ram_origin} who focused on the case when $H$ is a cycle.
As with the constrained Ramsey problem, we consider the question of finding a threshold for $G(n,p)$ having the anti-Ramsey property for a given $H$. 
\begin{question}[Anti-Ramsey]\label{ques:antiramsey}
    Let $H$ be a graph. Do there exist constants $c, C>0$ such that \[\lim\limits_{n \to \infty} \mathbb{P}[G(n,p) \antiramsey H] =   
\begin{cases}
    0       & \quad \text{if } p \leq cn^{-1/m_2(H)},\\
    1       & \quad \text{if } p \geq Cn^{-1/m_2(H)}.
\end{cases}\]  
\end{question}
As before, we refer to the two parts of this question as the \emph{0-statement for anti-Ramsey} and \emph{1-statement for anti-Ramsey} respectively.

 Note that, when $r = 2$, Theorem~\ref{thm:1-statement} is precisely the 1-statement of the anti-Ramsey problem, and so immediately yields a positive answer for the $1$-statement in Question~\ref{ques:antiramsey}.

In \cite{npss}, Nenadov, Person, \v{S}kori\'{c} and Steger introduced a general framework for proving 0-statements in various settings, and applied this to show that the $0$-statement for anti-Ramsey holds for sufficiently long cycles and sufficiently large complete graphs. 
These results were extended in \cite{anti-ram_cycles} for cycles and in \cite{anti-ram_complete} for complete graphs to obtain a 0-statement for anti-Ramsey for all cycles and complete graphs on at least four vertices.

The threshold here is $n^{-1/m_2(H)}$ for all cases except for when $H$ is a copy of $C_4$ or $K_4$, where it is instead given by the threshold for the appearance of a graph which has the anti-Ramsey property for $H$. 
Notice that if there exists a finite graph $G$ such that $m(G) < m_2(H)$ and $G \antiramsey H$, then for any $p = \omega(n^{-1/m(G)})$ (in particular for $p=n^{-1/m_2(H)}$) $G$ occurs in $G(n,p)$ w.h.p, and so $G(n,p) \antiramsey H$ w.h.p.
Therefore $n^{-1/m_2(H)}$ cannot be a threshold, and trivially a $1$-statement holds at $n^{-1/m(G)} = o(n^{-1/m_2(H)})$. 
Such graphs $G$ exist for $C_4$ and $K_4$. 

Similarly, if $H$ is a copy of $K_3$, then $H$ is rainbow in any proper colouring of $G(n,p)$. It follows that the threshold for $G(n,p)$ having the anti-Ramsey property for $K_3$ corresponds with the threshold for the appearance of a copy of $K_3$ which is at $n^{-1/m(K_3)} = n^{-1}$ (see \cite{bo}). 
This idea was pursued further in \cite{anti-ram_non-balanced} where it was demonstrated that $n^{-\frac{1}{m_2(H)}}$ is not a threshold for the anti-Ramsey property for graphs of the form $B_t \oplus F$:
here $B_t$ is the book graph formed by $t$ triangles all sharing a common edge, $F$ is any graph satisfying $1<m_2(F)<2$, and $\oplus$ is the graph operation of gluing two graphs together along an edge.

For the anti-Ramsey problem, we have some partial progress, reducing proving the $0$-statement to a (necessary) colouring statement. 
By the reasoning in the previous paragraph,
in order to prove that $n^{-1/m_2(H)}$ is a semi-sharp threshold, it is necessary to prove that
for all graphs $G$ with $m(G) \leq m_2(H)$, we have $G \notantiramsey H$.
For a coarse threshold one may replace $\leq$ by $<$ in the above statement.
We prove that it is in fact sufficient to prove the above colouring statement, for most graphs $H$.

\begin{theorem}\label{thm:antiramsey0}
    Let $H$ be a strictly $2$-balanced graph on at least five edges with $m_2(H)> 1$. To prove the $0$-statement of \Cref{ques:antiramsey} it suffices to prove for all graphs $G$ with $m(G) \leq m_2(H)$ that $G \notantiramsey H$. 
\end{theorem}

Note that when attempting to confirm the $0$-statement for anti-Ramsey, it is sufficient to consider the case when $H$ is strictly $2$-balanced. 
Indeed, otherwise, take a minimal subgraph $H' \subseteq H$ such that $m_2(H) = d_2(H')$, so that $H'$ is strictly 2-balanced.
If a proper colouring of $G(n,p)$ contains no rainbow copy of $H'$, it also contains no rainbow copy of $H$.

We prove Theorem~\ref{thm:antiramsey0} in Subsection~\ref{sec:anti-Ram_reduction}. Note that the only strictly $2$-balanced graphs excluded by the theorem are $K_3$, $C_4$ or a `cherry' $P_3$ (the path with two edges). For all of these cases, the colouring statement within is not true anyway 
(see \cite{anti-ram_cycles} for $C_4$, and notice that it fails trivially for $K_3$ and $P_3$ because every proper colouring of either of these graphs is rainbow), and the threshold is instead given by the appearance of a certain graph with the anti-Ramsey property. 

We prove the colouring statement for some special cases of $H$, including when $H$ is a $d$-regular graph on at least $4d$ vertices, in Subsection~\ref{sec:anti-Ram_special}. However, Question \ref{ques:antiramsey} remains open in general.

\subsection*{Organisation}
We prove Theorems~\ref{thm:constrained} and~\ref{thm:constrained_K3} in Section~\ref{sec:constrained} and we prove Theorem~\ref{thm:antiramsey0} in Section~\ref{sec:anti}.
In Section~\ref{sec:conclusion}, we give tables which provide a list of all known results for the constrained Ramsey and anti-Ramsey thresholds,
addressing what type of threshold occurs in each case, and we briefly discuss the remaining open cases for the anti-Ramsey threshold. 

\section{Anti-Ramsey results}\label{sec:anti}

In this section we prove \Cref{thm:antiramsey0} which reduces the $0$-statement of \Cref{ques:antiramsey} to a colouring statement (see \Cref{sec:anti-Ram_reduction}). 
We then apply it to give a family of graphs $H$ which satisfy the $0$-statement of the question (see \Cref{sec:anti-Ram_special}).

\subsection{Reduction to the colouring statement}
\label{sec:anti-Ram_reduction}

As preparation for the proof of \Cref{thm:antiramsey0}, we list some well-known facts regarding strictly $2$-balanced graphs.

\begin{fact}\label{fact:strictly2balancedprops}
    Let $H$ be a strictly $2$-balanced graph with $m_2(H) > 1$. Then the following hold.
    \begin{itemize}
        \item $H$ has minimum degree at least $2$.
        \item $H$ is $2$-connected (see \cite[Lemma~3.3]{ns}). 
    \end{itemize}
\end{fact}

We call a graph $H$ \emph{spacious} if for every edge $e$ in $H$ there exists an edge $f$ which is vertex disjoint from $e$. 
The following was observed by R\"{o}dl and Ruci\'{n}ski in \cite{rrlower}. We prove it here for completeness.

\begin{lemma}\label{lem:spacious}
    Let $H$ be a strictly $2$-balanced graph satisfying $m_2(H) > 1$ and $H \neq K_3$. Then $H$ is spacious.
\end{lemma}

\begin{proof}
	Assume for a contradiction that $H$ is not spacious. Then there is an edge $xy$ such that every other edge contains $x$ or $y$. As $m_2(H) > 1$, $H \not= K_2$.
	We claim that every vertex $z \in V(H) \setminus \{x,y\}$ is adjacent to both $x$ and $y$ and to no other vertices. Indeed, it has no neighbours other than $x$ and $y$ by assumption on $xy$, and it has degree at least $2$ by \Cref{fact:strictly2balancedprops}.
	Hence, $H$ contains a triangle (namely $xyz$ for any $z \in V(H) \setminus \{x,y\}$) and $e(H) = 1 + 2(v(H) - 2)$, showing that
	\begin{equation*}
		d_2(H) = \frac{1 + 2(v(H) - 2) - 1}{v(H) - 2} = 2 = m_2(K_3).
	\end{equation*}
	This contradicts the assumption that $H$ is strictly $2$-balanced and is not a triangle.
\end{proof}

Combining \Cref{fact:strictly2balancedprops} and \Cref{lem:spacious} with some of the proof of \cite[Theorem~7]{npss}, we will be able to prove \Cref{thm:antiramsey0}. 
In order for the proof to make sense, we require a number of definitions from the beginning of \cite[Section~2]{npss}.
Since we only deal with graphs -- and not hypergraphs -- in this paper, we will only state these definitions for graphs.

\begin{definition}[$H$-equivalence]
    Given graphs $H$ and $G$, say that two edges $e_1, e_2 \in E(G)$ are \emph{$H$-equivalent}, with notation $e_1 \equiv_H e_2$, if for every copy $H'$ of $H$ in $G$ we have $e_1 \in E(H')$ if and only if $e_2 \in E(H')$.
\end{definition}

\begin{definition}[\emph{$H$-closed} property]
    For given graphs $H$ and $G$, define the property of being \emph{$H$-closed} as follows:

    \begin{itemize}
        \item an edge $e\in E(G)$ is \emph{$H$-closed} if $e$ belongs to at least two copies of $H$ in $G$,
        \item a copy $H'$ of $H$ in $G$ is \emph{$H$-closed} if at least three edges from $E(H')$ are $H$-closed,
        \item a graph $G$ is \emph{$H$-closed} if every vertex and edge of $G$ belongs to at least one copy of $H$ and every copy of $H$ in $G$ is $H$-closed.
    \end{itemize}

\end{definition}

If the graph $H$ is clear from the context, we simply write \emph{closed}.

\begin{definition}[\emph{$H$-blocks}] 
	Given graphs $H$ and $G$, we say that $G$ is an \emph{$H$-block} if $G$ is $H$-closed and for every non-empty proper subset of edges $E' \subsetneq E(G)$ there exists a copy $H'$ of $H$ in $G$ such that $E(H') \cap E' \neq \emptyset$ and $E(H')\setminus E' \neq \emptyset$ (in other words, there exists a copy of $H$ which partially lies in $E'$).
\end{definition}

With these definitions we can now give the graph version of \cite[Theorem~12]{npss} which we will need in order to prove \Cref{thm:antiramsey0}.

\begin{theorem}[{\cite[Theorem~12]{npss}}]\label{thm:blocksarefinite}
    Let $H$ be a strictly $2$-balanced graph. Then there exist constants $c,L > 0$ such that for $p\leq cn^{-1/m_2(H)}$, w.h.p.\ $G \sim G(n,p)$ satisfies that every $H$-block $B \subseteq G$ contains at most $L$ vertices.
\end{theorem}

We will also need the following corollary of \Cref{thm:blocksarefinite}.

\begin{corollary}[{\cite[Corollary~13]{npss}}]\label{cor:blocksaresparse}
    Let $H$ be a strictly $2$-balanced graph.
    Then there exists a constant $c > 0$ such that for $p \leq cn^{-1/m_2(H)}$, w.h.p.\ $G \sim G(n,p)$ satisfies that for every $H$-block $B \subseteq G$ we have $m(B) \leq m_2(H)$.
    Moreover, if $p = o\left(n^{-1/m_2(H)}\right)$ then a strict inequality holds.
\end{corollary}

Further, we need the following basic property of $H$-closed graphs.
Here a \emph{decomposition} of a graph $G$ is a collection $H_1, \ldots, H_k$ of subgraphs of $G$ such that every edge of $G$ is covered by exactly one of the subgraphs $H_i$.

\begin{lemma}[{\cite[Lemma~14]{npss}}]\label{lem:fclosedpartition}
    Let $H$ be a graph. Then if a graph $G$ is $H$-closed, there exists a decomposition $B_1, \ldots, B_k$ of $G$, such that every subgraph $B_i$ is an $H$-block and every copy of $H$ in $G$ is entirely contained in some block $B_i$.
\end{lemma}

Our contribution to the proof of \Cref{thm:antiramsey0} is the following claim. 
It essentially says that removing two edges from a strictly $2$-balanced graph $H$ on at least five edges does not yield a graph that is `too far' from being spacious.

\begin{proposition}\label{prop:superspacious}
    Let $H$ be a strictly $2$-balanced graph on at least five edges with $m_2(H) > 1$. Then if we remove any two edges from $E(H)$, the resulting graph contains two edges which are vertex disjoint.
\end{proposition}

\begin{proof}
    Recall from \Cref{fact:strictly2balancedprops} and \Cref{lem:spacious} that $H$ has minimum degree at least $2$, $H$ is $2$-connected and $H$ is spacious. Denote the removed edges by $e$ and $f$. We have the following cases.

    \emph{Case 1}: \emph{The edges $e$ and $f$ are vertex disjoint}. In this case, if an edge $g$ is incident to both $e$ and $f$ in $H$ then we are immediately done since $H$ is spacious; 
    there is some edge in $H$ which is vertex disjoint from $g$, and is thus not $e$ or $f$. Assume there is no such edge $g$. 
    Then since $\delta(H) \geq 2$, there exists a collection $E$ of four distinct edges, each containing exactly one vertex of $V(e) \cup V(f)$. 
    Now, there are either two edges in $E$ which are vertex disjoint, and we are done, or all the edges in $E$ intersect in the same vertex $v$. 
	Since $H$ is $2$-connected there exists some edge $h$ which is not incident to $v$ and not in $E \cup \{e, f\}$ (otherwise $e$ and $f$ lie in different components of $H \setminus \{v\}$), and $h$ must be vertex disjoint from some edge from $E$.

    \emph{Case 2}: \emph{The edges $e$ and $f$ form a path}. 
    In this case, again, if an edge intersects both $e$ and $f$ then we are done by $H$ being spacious. 
    Assume not. 
    Then, since $\delta(H) \geq 2$, there are edges $e'$ and $f'$ such that $e'$ is incident to the vertex in $e \setminus f$ and $f'$ is incident to the vertex in $f \setminus e$.
	
    If $e'$ and $f'$ are vertex disjoint, we are done. Assume not.
    Then $e'$ and $f'$ meet at some vertex $v$ forming a copy of $C_4$ with $e$ and $f$. 
    Since $H$ has at least five edges, there exists at least one more edge $h$, which by assumption is not adjacent to $e\cap f$ and does not form a triangle with $e$ and $f$.
    If $h$ is vertex disjoint from either $e'$ or $f'$ we are done. Hence $h$ contains $v$.
	Observe that, since $H$ is $2$-connected, there also exists an edge which is neither incident to $v$ nor is one of $e$ or $f$ (otherwise $h \setminus \{v\}$ and $e$ are in different components of $H \setminus \{v\}$). 
    This edge is vertex disjoint from at least one of $e'$ and $f'$.
\end{proof}

We now adapt the proof of \cite[Theorem~7]{npss} to prove \Cref{thm:antiramsey0}. 

\begin{proofofthmantiramsey}
    Let $H$ be a strictly $2$-balanced graph on at least five edges and $c$ be a constant given by \Cref{cor:blocksaresparse} when applied to $H$. 
    Let $p \leq cn^{-1/m_2(H)}$ and $G \sim G(n,p)$. 
    Then w.h.p.\ every $H$-block $B$ in $G$ satisfies $m(B) \le m_2(H)$; let us assume that this holds for $G$. 
    We use Algorithm~\textsc{Rainbow-Colour} (see \Cref{fig:rainbowcolour}) 
    to find a proper colouring of $G$ without a rainbow copy of $H$. 

\begin{figure}[!ht]
\begin{algorithmic}[1]
\Procedure{Rainbow-colour}{$G = (V,E)$}
    \State $\hat{G}\gets G$
    \State col $\gets 0$
    \While{$\exists e_1, e_2 \in E(\hat{G}): e_1 \equiv_H e_2$ in $\hat{G}$ and $e_1 \cap e_2 = \emptyset$}\label{line:e1e2}
        \State colour $e_1, e_2$ with col\label{line:coloure1e2withcol}
        \State $\hat{G} \gets \hat{G}\setminus \{e_1, e_2\}$ and $\mbox{col} \ \gets \ \mbox{col} \ + 1$ \label{line:removee1e2iteratecol}
    \EndWhile\label{line:while1end}
    \While{$\exists e \in E(\hat{G})$: $e$ does not belong to a copy of $H$}\label{line:enotinFcopy}
        \State colour $e$ with col\label{line:colourewithcol}
        \State $\hat{G} \gets \hat{G}\setminus \{e\}$ and $\mbox{col} \ \gets \ \mbox{col} \ + 1$ \label{line:removeeiteratecol}
    \EndWhile\label{line:while2end}
\EndProcedure
\State Remove isolated vertices in $\hat{G}$.\label{line:removeisolated}
\State $\{B_1, \ldots, B_k\} \gets H$-blocks obtained by applying \Cref{lem:fclosedpartition} to $\hat{G}$.\label{line:applyblocklem}
\State Colour (properly) each $B_j$ without a rainbow copy of $H$ using distinct sets of colours (cf.\ text why these last two lines are possible).\label{line:finishcolouring}
\end{algorithmic}
\caption{The implementation of algorithm \textsc{Rainbow-colour}.}\label{fig:rainbowcolour}
\end{figure}

To see the correctness of the algorithm, we first deduce that it suffices to argue that the graph $\hat{G}$ obtained in \cref{line:removeisolated} can be properly coloured without creating a rainbow copy of $H$.
Indeed, to obtain $\hat{G}$ we only remove edges that are either in pairs of non-adjacent edges that are both contained in exactly the same $H$-copies (and can thus not be contained in a rainbow copy if we give them the same colour), or not contained in a copy of $H$ (and can thus be coloured with a new colour without risk of creating a rainbow copy of $H$).

It thus remains to prove that lines~\ref{line:applyblocklem} and~\ref{line:finishcolouring} are indeed possible. For the former, we must show that the graph $\hat{G}$ is $H$-closed.
Assume otherwise. Then there exists a copy $H'$ of $H$ which has at most two closed edges (as there are no vertices and edges which are not a part of a copy of $H$). 
But by applying \Cref{prop:superspacious} (and ensuring the closed edges are chosen to be removed), $H'$ must also have two edges $e_1, e_2$ which are vertex disjoint and are not closed. So $H'$ is the only copy of $H$ that $e_1$ and $e_2$ belong to and hence $e_1 \equiv_H e_2$.
This is a contradiction, as this pair of edges would have been removed in \cref{line:removee1e2iteratecol} of Algorithm~\ref{fig:rainbowcolour} (noting that the procedure in Lines~\ref{line:enotinFcopy} to \ref{line:while2end} does not change the collection of $H$-copies in $\hat{G}$). 
Hence $\hat{G}$ is indeed $H$-closed and thus, by assumption, every $H$-block $B$ in $G$ satisfies $m(B) \le m_2(H)$.

Now for line~\ref{line:finishcolouring}, by \Cref{lem:fclosedpartition} and our initial assumptions, colouring one block $B_i$ does not influence the colouring of any copy of $H$ which does not lie in $B_i$ and all  blocks $B_i$ are edge-disjoint.
It follows that for line~\ref{line:finishcolouring} to succeed (and to prove the $0$-statement of \Cref{ques:antiramsey}) it suffices to prove that all graphs $G$ with $m(G) \leq m_2(H)$ satisfy $G \notantiramsey H$.\qed\end{proofofthmantiramsey}

\subsection{Anti-Ramsey threshold in a special case}
\label{sec:anti-Ram_special}

We now wish to use Theorem~\ref{thm:antiramsey0} to find the threshold for the anti-Ramsey property for a class of graphs including sufficiently sparse regular graphs.

\begin{lemma}\label{lem:antiramsey_specialcase_colouring}
    Let $H$ be a strictly 2-balanced graph where $1 < m_2(H) < \frac{\delta(H)(\delta(H) + 1)}{2\delta(H) + 1}$. 
    Every graph $B$ such that $m(B) \le m_2(H)$ satisfies $B \notantiramsey H$.
\end{lemma}
\begin{proof}
    
    Suppose for a contradiction that there exists some graph $B$ with $m(B) \le m_2(H)$ and $B \antiramsey H$. Let $B$ be such a graph with the minimum number of vertices. 
    By minimality, every vertex $v \in B$ must be contained in a copy of $H$: 
    if not, we could take a proper colouring of $B\setminus \{v\}$ that does not contain a rainbow copy of $H$ and extend it to a proper colouring of $B$ with no rainbow copy of $H$ by colouring all the edges from $v$ with distinct new colours. 
    In particular, $\delta(B)\ge \delta(H)$.

    \begin{claim}
      $B$ contains two adjacent vertices of degree $\delta(H)$. 
    \end{claim}
    \begin{proof}[Proof of claim]
        Suppose not, for a contradiction. Let $X$ be the set of vertices of degree $\delta(H)$; by assumption, $X$ is a (possibly empty) independent set. We have 
        \[
        2e(B) \ge |X|\delta(H) + (v(B) - |X|)(\delta(H)+1) = v(B)(\delta(H)+1) - |X|
        \]
        and so $|X| \ge v(B)(\delta(H)+1) - 2e(B)$.
        Moreover, as $X$ is independent, 
        \[
        e(B) \ge |X|\delta(H) \ge 
       \big( v(B)(\delta(H)+1) - 2e(B)\big)\delta(H).
        \]
        Rearranging, we see that 
        \[
        m_2(H) \ge m(B) \geq \frac{e(B)}{v(B)} \ge \frac{\delta(H)(\delta(H) + 1)}{2\delta(H) + 1},
        \]
which contradicts that $m_2(H) < \frac{\delta(H)(\delta(H) + 1)}{2\delta(H) + 1}$.
    \end{proof}

    Therefore, we can take $u,v$ to be two adjacent vertices  of degree $\delta(H)$. 
    As $B$ is minimal, there is a proper colouring $\varphi$ of $B\setminus\{u,v\}$ that does not contain a rainbow copy of $H$. As $B \antiramsey H$, any extension of $\varphi$ to a proper colouring $\varphi'$ of $B$ contains a rainbow copy of $H$. 

    The only copies of $H$ that could be rainbow in such a $\varphi'$ must contain $\{u,v\} \cup N(u) \cup N(v)$  (as $d(u) = d(v) = \delta(H)$). So if there exist $x \not= y$ such that $ux, vy \in E(B)$, then using colours distinct from $\varphi$, colouring $ux$ and $vy$ the same and giving every other edge a distinct colour will yield $\varphi'$ with no rainbow copy of $H$. Therefore it must be the case that $N(u)\setminus v = N(v)\setminus u = \{w\}$, for some vertex $w$. However, by Fact~\ref{fact:strictly2balancedprops} and the conditions on $H$, we have that $H$ is 2-connected, and thus we must have that $H = K_3$. However, $K_3$ does not satisfy the condition that $m_2(H) < \frac{\delta(H)(\delta(H) + 1)}{2\delta(H) + 1}$, giving our desired contradiction.
\end{proof}

Theorem~\ref{thm:antiramsey0} and Lemma~\ref{lem:antiramsey_specialcase_colouring} immediately imply the following.

\begin{theorem}\label{thm:antiramsey_specialcase}
     Let $H$ be a strictly 2-balanced graph where $1 < m_2(H) < \frac{\delta(H)(\delta(H) + 1)}{2\delta(H) + 1}$. 
     There exists a constant $c$ such that if $p \le cn^{-1/m_2(H)}$ then
     \[
     \lim_{n\rightarrow \infty} \mathbb{P}[G(n,p)\antiramsey H] = 0.
     \]
\end{theorem}

As a corollary, we get that every relatively sparse, strictly $2$-balanced regular graph satisfies the $0$-statement in \Cref{ques:constrainedramsey}.

\begin{corollary} \label{cor:antiramsey-regular}
    Let $H$ be a $d$-regular strictly 2-balanced graph with $v(H) \ge 4d$. 
    There exists a constant $c$ such that if $p \le cn^{-1/m_2(H)}$ then
     \[
     \lim_{n\rightarrow \infty} \mathbb{P}[G(n,p)\antiramsey H] = 0.
     \]
\end{corollary}
\begin{proof}
    Since $H$ is $d$-regular, $e(H) = dv(H)/2$. Then 
    \[
    m_2(H) = \frac{e(H) - 1}{v(H) - 2}  = \frac{d}{2} + \frac{d-1}{v(H)-2}.
    \]
    Rearranging, we find that $m_2(H) <  \frac{d}{2} + \frac{d}{4d+2} = \frac{d(d+1)}{2d+1}$ if 
    $v(H) > 4d - \frac{2}{d}$. Thus we can apply Theorem~\ref{thm:antiramsey_specialcase} to deduce the result.    
\end{proof}

As a concrete example of a family of graphs $H$ satisfying the conditions of \Cref{cor:antiramsey-regular}, consider for $k \in \mathbb{N}$ the family of $k$-blow-ups of cycles of length at least $8$.

\section{Constrained Ramsey}
\label{sec:constrained}

In this section we will prove the 0-statements of Theorems~\ref{thm:constrained} and \ref{thm:constrained_K3}, regarding the constrained Ramsey problem for $(K_{1,k},H)$, where $k \ge 3$. 
(Recall that the 1-statement of both theorems follows from Theorem~\ref{thm:1-statement}.)
The proof of the 0-statements splits into two parts, according to whether $H$ is a triangle or not. 
\footnote{Technically, the two parts depend on whether there is a strictly $2$-balanced graph $J \ne K_3$ with $m_2(J)=m_2(H)=2$ contained in $H$.}

\subsection{Threshold when $H$ is not a triangle}
	In this subsection we prove the following lemma, which implies the $0$-statement of Theorem~\ref{thm:constrained} in the case where $H$ is not a triangle. 
    Recall that\footnote{See just after the statement of Theorem~\ref{thm:antiramsey0} on page \pageref{thm:antiramsey0}.} we may assume $H$ is strictly $2$-balanced. 

	\begin{lemma}\label{lem:notk3}
		Let $H$ be a strictly $2$-balanced graph such that $m_2(H) > 1$ and $H \neq K_3$. There exists a constant $c>0$ such that if $p \leq cn^{-1/m_2(H)}$ then
		\[\lim\limits_{n \to \infty} \mathbb{P}[G(n,p) {\notconstrained} (K_{1,3}, H)] = 0.\]
	\end{lemma}

	Note that the same conclusion with $K_{1,3}$ replaced with $K_{1,k}$ for any fixed integer $k \ge 3$ follows trivially from \Cref{lem:notk3}. We will make use of the following result from \cite{npss}. 
    Note that the original result was proven in more general terms regarding so-called $2$-bounded colourings (these are colourings where each colour is used on at most two edges); here we will only need the weaker constrained Ramsey property for $(K_{1,3}, H)$.

	\begin{lemma}[{\cite[Lemma~26]{npss}}]\label{lem:notc4}
		Let $H$ be a strictly $2$-balanced graph on at least four vertices with $m_2(H) > 1$ such that $H \neq C_4$. Then for any graph $B$ such that $m(B) \leq m_2(H)$ it holds that $B \notconstrained (K_{1,3}, H)$. 
	\end{lemma}

	For $H = C_4$, since we are working with the $(K_{1,3}, H)$ constrained Ramsey property and not $2$-bounded colourings, we can prove a similar lemma.\footnote{Note that there does exist a graph $B$ such that $m(B) = m_2(C_4)$ and any 2-bounded colouring of $B$ contains a rainbow $C_4$ (see the proof of \cite[Lemma~27]{npss}).}  Its proof uses Vizing's theorem, that is, $\chi'(G) \in \{\Delta(G), \Delta(G) + 1\}$. 


	\begin{lemma}\label{lem:c4}
		Every graph $B$ such that $m(B) \leq m_2(C_4)$ satisfies $B \notconstrained (K_{1,3}, C_4)$.
	\end{lemma}

	\begin{proof}
		Assume for a contradiction that $B$ is a minimal graph on $n$ vertices such that $m(B) \le m_2(C_4) = 3/2$ and $B \constrained (K_{1,3}, C_4)$. 

		Suppose that there is a vertex $v$ in $B$ with degree at most $2$.
		By the minimality assumption on $B$, colour $B\setminus\{v\}$ to witness that $B\setminus\{v\} \notconstrained (K_{1,3}, C_4)$ and colour the (at most) two edges incident to $v$ with the same (new) colour. This yields a colouring of $B$ with no rainbow copy of $C_4$ and no monochromatic copy of $K_{1,3}$, contradicting our choice of $B$.

		Thus $B$ has minimum degree at least $3$, so $e(B) \ge 3n/2$.
		Since $e(B) \le m(B)n \le 3n/2$, we find that $e(B) = 3n/2$, implying that $\Delta(B) = 3$.

        Hence by Vizing's theorem, $\chi'(G) \in \{3,4\}$. If $\chi'(G) = 3$, then we have a proper edge colouring on $3$ colours, and thus trivially no rainbow copies of $C_4$ or monochromatic copies of $K_{1,3}$. If $\chi'(G) = 4$, set any edges of colour $4$ to be colour $3$. Then we have an edge colouring on $3$ colours and thus no rainbow copies of $C_4$. Moreover, we cannot have created a monochromatic copy of $K_{1,3}$ as the original 4 edge-colouring of $G$ was proper.
	\end{proof}

	We are now in a position to prove Lemma~\ref{lem:notk3}. Our proof will largely follow that of \cite[Theorem~5]{npss}. \begin{figure}[!ht]
		\begin{algorithmic}[1]
		\Procedure{Rainbow-colour-constrained}{$G = (V,E)$}
			\State $\hat{G}\gets G$
			\State col $\gets 0$
			\While{$\exists \text{ distinct } e_1, e_2 \in E(\hat{G}): e_1 \equiv_H e_2$ in $\hat{G}$}\label{line:e1e2con}
				\State colour $e_1, e_2$ with col\label{line:coloure1e2withcolcon}
				\State $\hat{G} \gets \hat{G}\setminus \{e_1, e_2\}$ and $\mbox{col} \ \gets \ \mbox{col} \ + 1$ \label{line:removee1e2iteratecolcon}
			\EndWhile\label{line:while1endcon}
			\While{$\exists e \in E(\hat{G})$: $e$ does not belong to a copy of $H$}\label{line:enotinFcopycon}
				\State colour $e$ with col\label{line:colourewithcolcon}
				\State $\hat{G} \gets \hat{G}\setminus \{e\}$ and $\mbox{col} \ \gets \ \mbox{col} \ + 1$ \label{line:removeeiteratecolcon}
			\EndWhile\label{line:while2endcon}
		\EndProcedure
		\State Remove isolated vertices in $\hat{G}$.\label{line:removeisolatedcon}
		\State $\{B_1, \ldots, B_k\} \gets H$-blocks obtained by applying \Cref{lem:fclosedpartition} on $\hat{G}$.\label{line:applyblocklemcon}
		\State Colour (properly) each $B_j$ without a rainbow copy of $H$ using distinct sets of colours (cf.\ text why the last two lines are possible)\label{line:finishcolouringcon}.
		\end{algorithmic}
		\caption{The algorithm \textsc{Rainbow-colour-constrained}.}\label{fig:constrainednotk3}
		\end{figure}

	\begin{proof}[Proof of Lemma~\ref{lem:notk3}.]
		Let $c$ be a constant given by \Cref{cor:blocksaresparse} when applied to $H$. Let $G \sim G(n,p)$ for $p \leq cn^{-1/m_2(H)}$. Then w.h.p.\ every $H$-block $B$ in $G$ satisfies $m(B) \le m_2(H)$. Assume this is indeed the case. We use Algorithm~\textsc{Rainbow-colour-constrained} (above) to find a colouring containing no monochromatic copy of $K_{1,3}$ and no rainbow copy of $H$. 
		Notice that the only difference between Algorithms~\textsc{Rainbow-colour} and \textsc{Rainbow-colour-constrained} is in the condition in line~\ref{line:e1e2con}. We may proceed as in the proof of \Cref{thm:antiramsey0}, noting that $\hat{G}$ is $H$-closed by construction.   Indeed, every edge and vertex of $\hat{G}$ is contained in a copy of $H$ by lines~\ref{line:enotinFcopycon} and \ref{line:removeisolatedcon}, and line~\ref{line:e1e2con} implies that every copy of $H$ must contain at least three closed edges, namely the (at least three) edges that are in at least two copies of $H$ (because otherwise, using $e(H) \ge 4$, there is a copy $H'$ of $H$ with at least two edges that do not appear in other copies of $H$, contradicting line~\ref{line:e1e2con}). 
        Now $\hat{G}$ being $H$-closed means line~\ref{line:applyblocklemcon} succeeds, and so by applying \Cref{lem:notc4,lem:c4} we can colour each $H$-block $B_i$ as desired.
        We do so by using disjoint sets of colours (that also avoid the colours already used by the algorithm) for different blocks, and hence w.h.p.\ Algorithm~\textsc{Rainbow-colour-constrained} finds the desired colouring of $G$.
	\end{proof}

\subsection{Reduction to colouring statement when $H$ is a triangle}

	We cannot prove the 0-statements of Theorems~\ref{thm:constrained} for $H = K_3$ and Theorem~\ref{thm:constrained_K3} using a similar strategy to the proof of Theorem~\ref{thm:antiramsey0}. 
    Indeed, one cannot apply \Cref{prop:superspacious} since it is only applicable to graphs with at least five edges. Furthermore, every edge of a $K_3$-block is closed, i.e.\ contained in at least two $K_3$-copies, a property too strong to aim for in a proof.
	
    Instead of considering $K_3$-blocks, we will prove the 0-statements of Theorems~\ref{thm:constrained} for $H = K_3$ and Theorem~\ref{thm:constrained_K3} by considering a broader class of structures we will call \emph{triangle-connected} graphs. 
 
    We need the following definition. A \emph{triangle sequence} $\vecT$ is a sequence $T_0, \ldots, T_{\ell}=T$ such that:
	\begin{itemize}
		\item
			$T_0$ is a triangle.
		\item
			For $i \in [\ell]$, there is a vertex $v_i$ in $V(T) \setminus V(T_{i-1})$ and an edge $e_i$ in $T_{i-1}$ that together form a triangle in $T$, such that $V(T_i) = V(T_{i-1}) \cup \{v_i\}$ and $E(T_i)$ is the union of $E(T_{i-1})$ with all edges of $T$ between $v_i$ and $V(T_{i-1})$ (which in particular includes the two edges between $v_i$ and $e_i$).
	\end{itemize} 

    Further, we say that $i$ is a \emph{regular step} if $\deg(v_i, T_{i-1}) = 2$. 
	We say that a graph $T$ is \emph{triangle-connected} if $T$ has no isolated vertices, every edge in $T$ is in a triangle, and the $3$-uniform hypergraph on $V(T)$ whose edges are the triangles in $T$ is tightly-connected\footnote{A hypergraph is tightly-connected if it can be obtained by starting with a hyperedge and adding hyperedges one by one, such that every added hyperedge intersects with one of the previous hyperedges in 2 vertices.}. 

	Observe also that if $T$ is triangle-connected, then there is a triangle sequence $T_0, \ldots, T_{\ell}$ with $T_{\ell} = T$ (start with an arbitrary triangle $T_0$ and always add all edges between $v_i$ to $T_{i-1}$ for some $v_i$ which is in a triangle with an edge in $T_{i-1}$; 
    this process will only stop when we have exhausted all vertices of $T$, by triangle-connectivity).
	In the next lemma we show that if $p$ is sufficiently smaller than $n^{-1/2}$ then, w.h.p., every triangle-connected subgraph of $G(n,p)$ has low density.
	Using this, in order to prove the 0-statements of Theorem~\ref{thm:constrained} for $H$ being a triangle and Theorem~\ref{thm:constrained_K3}, it suffices to show that every triangle-connected graph with low density can be coloured appropriately (see \Cref{lem:k>3_triangle_colouring,lem:k=3_triangle_colouring}).

    \begin{lemma} \label{lem:reduction-to-colouring-triangle}
		\hfill
		\begin{itemize}
			\item
				If $p \le n^{-1/2}/20$ then, w.h.p., every triangle-connected subgraph $T$ of $G(n,p)$ satisfies $e(T) \le 2v(T)$.
			\item
				If $p = o(n^{-1/2})$ then, w.h.p., every triangle-connected subgraph $T$ of $G(n,p)$ satisfies $e(T) < 2v(T)$.
		\end{itemize}
	\end{lemma}

    Before we prove Lemma~\ref{lem:reduction-to-colouring-triangle}, we develop our understanding of triangle sequences.
		
	For a triangle sequence $\vecT = (T_0, \ldots, T_{\ell})$, define 
	\begin{equation} \label{eqn:rH}
		r(\vecT) = \sum_{i \in [\ell]} \big(\deg(v_i, T_{i-1}) - 2\big),
	\end{equation}
	and notice that $e(T_{\ell}) = 3 + 2\ell + r(\vecT)$ and $v(T_{\ell}) = \ell + 3$.
	
	We therefore have
	\begin{enumerate}[label = \ding{72}]
		\item \label{itm:r}
			If $T$ is triangle-connected and $e(T) \ge 2v(T)$ then there is a triangle sequence $\vecT = (T_0, \ldots, T_{\ell})$ with $T_{\ell} = T$ and $r(\vecT) \ge 3$. If $e(T) > 2v(T)$ then there is a triangle sequence $\vecT = (T_0, \ldots, T_{\ell})$ with $T_{\ell} = T$ and $r(\vecT) \ge 4$.
	\end{enumerate}

	Let $\rho$ be a positive integer. Say that a triangle sequence $\vecT = (T_0, \ldots, T_{\ell})$ is \emph{$\rho$-minimal} if $r(\vecT) \ge \rho$ and there is no triangle sequence $\vecT' = (T_0', \ldots, T_{\ell'}')$ with $r(\vecT') \ge \rho$ and $T_{\ell'}' \subsetneq T_{\ell}$.
	In the next claim we prove basic properties of $\rho$-minimal triangle sequences.

	\begin{proposition} \label{prop:rho-minimal}
		Let $\vecT = (T_0, \ldots, T_{\ell})$ be a $\rho$-minimal triangle sequence. Then
		\begin{enumerate}[label = (\alph*)]
			\item \label{itm:rho-min-exactly-rho}
				$r(\vecT) = \rho$.
			\item \label{itm:rho-min-irreg}
				There are at most $\rho$ irregular steps.
			\item \label{itm:rho-min-fi}
				For all but at most $\rho(\rho+2)$ regular steps $i$, there is an edge $f_i$ between $v_i$ and $e_i$, such that $f_i$ forms a triangle together with $v_j$, for some $j > i$.
		\end{enumerate}
	\end{proposition}

	\begin{proof}
		For \ref{itm:rho-min-exactly-rho}, notice that $r(\vecT) = r(T_0, \ldots, T_{\ell-1}) + \deg(v_{\ell}, T_{\ell-1}) - 2$.
        By minimality, we have $r(T_0, \ldots, T_{\ell-1}) < \rho$, showing that $\deg(v_{\ell}, T_{\ell-1}) \ge 3$.
        If $r(\vecT) > \rho$, then remove an edge between $v_{\ell}$ and $V(T_{\ell-1}) \setminus V(e_{\ell})$ to obtain a graph $T_{\ell}'$,
        and notice that $T_0, \ldots, T_{\ell-1}, T_{\ell}'$ is a triangle sequence with $r(T_0, \ldots, T_{\ell-1}, T_{\ell}') \ge \rho$, contradicting minimality.

		For \ref{itm:rho-min-irreg}, notice that every irregular step contributes at least $1$ to $r(\vecT)$, so by \ref{itm:rho-min-exactly-rho} there are at most $\rho$ irregular steps.

		Notice that the number of edges added during irregular steps is at most $\rho(\rho+2)$, using \ref{itm:rho-min-irreg} to see that there are at most $\rho$ irregular steps, and using \ref{itm:rho-min-exactly-rho} to see that at most $\rho+2$ edges are added during any single irregular step.
		Thus there are at most $\rho(\rho+2)$ regular steps $i$ such that $v_i$ is contained in an edge added during an irregular step.
  
		Finally, consider a regular step $i$ such that $v_i$ is not contained in an edge added during an irregular step. 
  
        \begin{claim}\label{claim:regular}
		  Let $i$ be a regular step such that $v_i$ is not contained in an edge added during any irregular step. Then there exists $j > i$ such that $v_i \in e_j$.
		\end{claim} \begin{proof} Suppose for contradiction that this fails for some regular step $i$. Define a new triangle sequence $\vecT' = (T_0', \ldots, T_{\ell-1}')$, where
		\begin{equation*}
			T_j' = \left\{
				\begin{array}{ll}
					T_j & \text{if $j < i$}, \\
					T_{j+1} \setminus \{v_i\} & \text{otherwise}.
				\end{array}
				\right.
		\end{equation*}
		Notice that $\vecT'$ is indeed a triangle sequence (using that no edge $e_j$ with $j > i$ in the original sequence involved the vertex $v_i$).
		Moreover, $r(\vecT') = r(\vecT)$, because $i$ is a regular step in the original sequence and is not contained in an edge added during any irregular step. That is, it contributes $0$ to the sum defining $r(\vecT)$ (see \eqref{eqn:rH}).
		The existence of the new sequence $\vecT'$ contradicts the minimality of $\vecT$, proving the claim. \qedhere \end{proof} 
  
        By Claim~\ref{claim:regular}, let $j > i$ be minimal such that $v_i \in e_j$. Then, right before step $j$, the vertex $v_i$ is incident only to the two edges joining it to $e_i$. This implies the existence of an edge $f_i$ as in \ref{itm:rho-min-fi}.
	\end{proof}

	In a $\rho$-minimal triangle sequence $T_0, \ldots, T_{\ell}$, for every regular step $i$ let $f_i$ be an edge as in \Cref{prop:rho-minimal}~\ref{itm:rho-min-fi}, if it exists. 
    Define $J(i) := \{j > i : \text{$v_j$ and $f_i$ forms a triangle}\}$, with $J(i) := \emptyset$ if such $f_i$ does not exist or if $i$ is an irregular step.
	Then by Claim~\ref{prop:rho-minimal}~\ref{itm:rho-min-fi} we have $J(i) \neq \emptyset$ for all but $\rho(\rho+2)$ regular steps $i$.
    In the following claim we show that for almost all regular $i$ we have that $|J(i)| = 1$, and moreover for almost all $i$, the edge $e_i$ is equal to $f_j$ for some $j < i$ satisfying $|J(j)| = 1$. 
    These properties will be used to count the number of non-isomorphic $\rho$-minimal triangle sequences.

	\begin{proposition} \label{prop:Ji}
		Let $T_0, \ldots, T_{\ell}$ be a $\rho$-minimal triangle sequence, with $\rho \le 4$. Then, the following holds.
		\begin{enumerate}[label = (\alph*)]
			\item \label{itm:rho-min-Jij}
				For every $i \in [\ell]$, for all but at most $200$ values of $i' < i$, it holds that $|J(i')| = 1$ and the single element $j$ in $J(i')$ satisfies $j < i$.
			\item \label{itm:rho-min-Ji}
				For all but at most $260$ regular steps $i$, we have $e_i = f_{j}$ for some $j < i$ such that $|J(j)| = 1$.
		\end{enumerate}
	\end{proposition}

	Note that for clarity of exposition, we make no attempt to optimise the constants in the lemmas of this section.

	\begin{proof}
		Let $x$ be the number of regular steps $i'$ with $|J(i')| \ge 2$.
		Then the number of triangles in $T_{\ell}$ is at least 
		\begin{equation*}
			\sum_{i \in [\ell]}|J(i)| \ge \#\{\text{regular steps}\} - \rho(\rho+2) + x \ge \ell - \rho - \rho(\rho+2) + x \ge \ell - 28 + x.
		\end{equation*}
		For the first inequality we used that every pair $(i,j)$ with $j \in J(i)$ gives rise to the triangle $f_i \cup \{v_j\}$, and these triangles are distinct for different pairs, and moreover we used that $J(i) \neq \emptyset$ for all but at most $\rho(\rho+2)$ regular steps $i$ by \Cref{prop:rho-minimal}~\ref{itm:rho-min-fi}.
		For the second inequality we used that there are at most $\rho$ irregular steps, by \Cref{prop:rho-minimal}~\ref{itm:rho-min-irreg}.
		The number of triangles is also at most
		\begin{equation*}
			1 + \ell + \rho \cdot \left(\binom{6}{2} - 1\right) \le \ell + 57,
		\end{equation*}
		because regular steps (and step $0$) give rise to exactly one new triangle each, and irregular steps, of which there are at most $\rho$, yield at most $\binom{6}{2}$ new triangles each.
		Altogether, we get that $x \le 85$.

		Fix $i \in [\ell]$. 
		Let $y$ be the number of regular steps $i' < i$ with $|J(i')| = 1$ for which the single element $j \in J(i')$ satisfies $j \ge i$.
		The number of triangles in $T_{\ell}$ which are not in $T_{i-1}$ is at least
		\begin{equation*}
			\sum_{i'}\big|J(i') \cap [i, \ell]\big|
			\ge y + \sum_{i' \ge i}|J(i')|
			\ge y + \ell - i - \rho - \rho(\rho+2)
			\ge \ell - i - 28 + y.
		\end{equation*}
		Here we used that every pair $(i',j)$ with $j \ge i$ and $j \in J(i')$ gives rise to the triangle formed by $f_{i'} \cup \{v_j\}$ which is not in $T_{i-1}$, and we also used that $|J(i')| \ge 1$ for all but at most $\rho(\rho+2)$ regular steps $i'$.
		The number of triangles in $T_{\ell}$ which are not in $T_{i-1}$ is also at most
		\begin{equation*}
			\ell - i + 1 + \rho \cdot \left(\binom{6}{2} - 1\right)
			\le \ell - i + 57,
		\end{equation*}
		where again we used that regular steps gives rise to a unique new triangle, and irregular steps yield at most $\binom{6}{2}$ new triangles.
		It follows that $y \le 85$.

		It is now easy to conclude \ref{itm:rho-min-Jij}. Indeed, every $i' < i$ satisfies one of the following: $i'$ is irregular; $|J(i')| \ge 2$; $|J(i')| = 1$ and the unique element $j$ in $J(i')$ satisfies $j \ge i$; $i'$ is regular and $J(i') = \emptyset$; or $|J(i')| = 1$ and the single element $j$ in $J(i')$ satisfies $j < i$.
		There are at most four irregular steps, at most $85$ steps of the second type, at most $85$ steps of the third type, at most $\rho(\rho+2) \le 24$ steps of the fourth type, and so there are at least $i - (4 + 85 + 85 + 24) = i - 198 \ge i - 200$ steps of the fifth type, as claimed in \ref{itm:rho-min-Jij}.

		For \ref{itm:rho-min-Ji}, notice that the number of regular $i$, for which there is $j < i$ with $e_i = f_j$ and $|J(j)| = 1$, is at least the number of triangles of the form $f_j \cup \{v_i\}$ with $|J(j)| = 1$ and $i \in J(j)$,
		minus the number of triangles of form $e \cup \{v_i\}$ with $i$ irregular and $e \in E(T_{i-1})$.
		The first quantity is at least $(\ell - 1) - 200$, by \ref{itm:rho-min-Jij} (with $i = \ell$). The second quantity is at most $\rho \cdot \binom{6}{2} \le 60$.
		It follows that the number of $i$ as in \ref{itm:rho-min-Ji} is at least $(\ell - 200) - 60 \ge \ell - 260$.
	\end{proof}

	We are now in a position to prove Lemma~\ref{lem:reduction-to-colouring-triangle}.
	
    \begin{proof}[Proof of Lemma~\ref{lem:reduction-to-colouring-triangle}]
		Let $m_{\ell,\rho}$ be the number of non-isomorphic $\rho$-minimal triangle sequences $T_0, \ldots, T_{\ell}$. Then, for $\ell \ge 1$,
		\begin{align*}
			m_{\ell, \rho} 
			& \le \binom{\ell}{\le 4} \cdot \binom{\ell}{\le 260} \cdot \binom{\ell + 3}{\le 6}^{\rho} \cdot (2\ell + 7)^{260} \cdot 200^{\ell} \\[.7em]
			& \le 2^{1 + 1 + 4 + 2 \cdot 6 \cdot 4 + 4 \cdot 260} \cdot \ell^{4+260+24+260} \cdot 200^{\ell} 
			\le 2^{1100} \cdot \ell^{600} \cdot 200^{\ell}.
		\end{align*}
		(Here $\binom{\ell}{\le i}$ stands for $\binom{\ell}{0} + \ldots + \binom{\ell}{i}$.)
		Indeed, in the first inequality, the term $\binom{\ell}{\le 4}$ bounds the number of ways to choose which steps are irregular,
        the term $\binom{\ell}{\le 260}$ bounds the number of ways to choose which regular steps $i$ do not satisfy $e_i = f_j$ for some $j < i$ with $|J(j)| = 1$ (using \Cref{prop:Ji}~\ref{itm:rho-min-Ji}), 
        the term $\binom{\ell + 3}{\le 6}^{\rho}$ bounds the number of ways to choose the neighbours of $v_i$ in $T_{i-1}$ for irregular $i$, 
        the term $(2\ell + 7)^{260}$ bounds the number of ways to choose the edge $e_i$ for regular $i$ for which it is not the case that $e_i = f_j$ with $|J(j)| = 1$ (using that $e(T_{\ell}) = 3 + 2\ell + r(\vecT) \le 2\ell + 7$), 
        and finally the term $200^{\ell}$ bounds the number of ways to choose $e_i$ for all other regular $i$ (using \Cref{prop:Ji}~\ref{itm:rho-min-Jij}).
		For the inequality we used $\binom{\ell}{\le i} \le \ell^i + 1 \le 2\ell^i$, $\ell + 3 \le 4\ell$, and $2\ell + 7 \le 16\ell$ for $\ell \ge 1$.

		We now complete the proof of the lemma. First, suppose that $p = o(n^{-1/2})$. 
        Then the expected number of $3$-minimal triangle sequences $T_0, \ldots, T_{\ell}$ with $T_{\ell} \subseteq G(n,p)$ is at most
		\begin{align*}
			\sum_{\ell \ge 1} m_{\ell,3} \cdot n^{\ell + 3} \cdot p^{2\ell+6}   
			\le \sum_{\ell \ge 1} 2^{1100} \cdot \ell^{600} \cdot \left(200 np^2\right)^{\ell+3} = o(1).
		\end{align*}
		Indeed, for the first inequality we used that $e(T_{\ell}) = 3 + 2\ell + r(\vecT) = 2\ell + 6$.
		Thus, w.h.p. $G(n,p)$ has no $3$-minimal triangle sequences, implying that there are no triangle sequences $\vecT$ with $r(\vecT) \ge 3$, 
        which in turn shows that there are no subgraphs $T$ which are triangle-connected and satisfy $e(T) \ge 2v(T)$ by \ref{itm:r}.

		Next, suppose that $p \le n^{-1/2}/20$. Then the expected number of $4$-minimal triangle sequences $T_0, \ldots, T_{\ell}$ with $T_{\ell} \subseteq G(n,p)$ is at most
		\begin{align*}
			\sum_{\ell \ge 1} m_{\ell,4} \cdot n^{\ell+3} \cdot p^{2\ell+7} 
			& \le p \cdot \sum_{\ell \ge 1} 2^{1100} \cdot \ell^{600} \cdot \left(200 np^2\right)^{\ell+3} \\
			& \le n^{-1/2} \cdot \sum_{\ell\geq 1} 2^{1100} \cdot \ell^{600} \cdot \left(\frac{200}{400}\right)^{\ell+3} = o(1).
		\end{align*}
		As before, we conclude that, w.h.p., $G(n,p)$ has no $4$-minimal triangle sequences, showing that there are no triangle sequences $\vecT$ with $r(\vecT) \ge 4$, which in turn shows that there are no subgraphs $T$ which are triangle-connected and satisfy $e(T) > 2v(T)$ by \ref{itm:r}.
	\end{proof}

    \subsection{Proof of colouring statements}
    Armed with Lemma~\ref{lem:reduction-to-colouring-triangle}, it suffices to prove the following two colouring statements to conclude the respective 0-statements of Theorem~\ref{thm:constrained} (for $H = K_3$) and Theorem~\ref{thm:constrained_K3}. 

	\begin{lemma} \label{lem:k>3_triangle_colouring}
		For any triangle-connected graph $T$ with $e(T)\le 2v(T)$ it holds that $T \notconstrained (K_{1,4},K_3)$.
	\end{lemma}

    \begin{lemma}\label{lem:k=3_triangle_colouring}
		For any triangle-connected graph $T$ with $e(T)< 2v(T)$ it holds that $T \notconstrained (K_{1,3},K_3)$.
	\end{lemma}
    
	\begin{proof}[Proof of Lemma~\ref{lem:k>3_triangle_colouring}]

		Let $B$ be a triangle-connected graph with $e(B) \le 2v(B)$ and let $\vecT = (T_0, \ldots, T_{\ell})$ be a triangle sequence such that $T_{\ell} = B$. Write $\rho = r(\vecT)$. Then $\rho \le 3$, as $e(B) \leq 2v(B)$, $e(B) = 3 + 2\ell + \rho$ and $v(B) = \ell + 3$.
		Assume that $\vecT$ is chosen so that the first irregular step (if one exists) appears as early as possible, and similarly that each subsequent irregular step (if any exist) appears as early as possible given this. 
        We now construct a sequence $\phi_0, \ldots, \phi_{\ell}$ of partial colourings of $T_0, \ldots, T_{\ell}$, such that:
        $\phi_i$ extends $\phi_{i-1}$ for $i \in [\ell]$ and uses the colour set $[0,i]$; every triangle in $T_i$ has at least two edges coloured with the same colour in $\phi_i$; and there is no monochromatic $K_{1,4}$.
		Note that this suffices to prove the lemma, as we may finish by colouring each edge in $B$ that is left uncoloured by $\phi_{\ell}$ with a unique new colour.

		\begin{itemize}
			\item
				For $i = 0$, the graph $T_0$ is a triangle. Colour two edges of $T_0$ by $0$, and leave the remaining edge uncoloured.
			\item
				For $i \geq 1$, let $U$ be the set of vertices $u'$ in $T_{i-1}$ such that there exists an edge $e \in E(T_{i-1})$ containing $u'$ and forming a triangle with $v_i$ (so $\delta(T_{i-1}[U]) \ge 1$). 
                \begin{itemize} 
                        \item
                        If  $|U| \le 3$ and there are at most three edges from $v_i$ to $T_{i-1}$ then colour these edges with colour $i$. If $|U| \le 3$ and $d(v_i, T_{i-1}) \ge 4$, colour all edges from $v_i$ to $U$ by colour $i$ (leaving the other edges from $v_i$ to $T_{i-1}$ uncoloured).

                        \item If $|U| \ge 4$, define $u$ to be the last vertex in $U$ added to the sequence $\vecT$.

						We claim that $|U| = 4$. Indeed, otherwise $|U| = 5$ (as $\rho \le 3$), and this is the first (and only) irregular step.
						Then $u$ was added in a regular step, so it sends at most two edges to $U$.
						This and $\delta(T_{i-1}[U]) \ge 1$ imply that $U \setminus \{u\}$ spans an edge. 
                        But then we could have added $v_i$ to the graph right before adding $u$ (and joining $v_i$ to $U \setminus \{u\}$), thus having the first irregular step appear earlier, a contradiction.

						It remains to consider the case $|U| = 4$.
						We claim that $u$ was added in an irregular step. 
                        Indeed, notice that by the assumption on the sequence, we may assume that $v_i$ was added right after $u$ (the only other possibility is that there was another irregular step right between $u$ and $v_i$, but then we can swap the order of these two irregular steps).
						Assuming $u$ was added in a regular step, $u$ sends at most two edges to $U$. 
                        This shows that $U \setminus \{u\}$ spans an edge, and thus $v_i$ could have been added before $u$ in an irregular step, a contradiction.

						So this step and the step $j$ in which $u$ was added are the only irregular steps, and $u$ was added along with exactly three edges, which are coloured $j$, and no other edges are coloured $j$.
						Colour the edges between $v_i$ and $U \setminus \{u\}$ by $j$.
						Then the colour class of $j$ has maximum degree at most $3$, and every triangle touching $v_i$ has two edges coloured $j$.
						\qedhere
				\end{itemize}
		\end{itemize}
	\end{proof}

    We now prove Lemma~\ref{lem:k=3_triangle_colouring}. The proof is similar to \Cref{lem:k>3_triangle_colouring} but more involved.
    
	\begin{proof}[Proof of Lemma~\ref{lem:k=3_triangle_colouring}]
        Let $B$ be a triangle-connected graph with $e(B) < 2v(B)$. Here we use a variant of a triangle sequence, where we allow $T_0$ to be either a triangle, a $K_4$, or a $K_5^-$ (a $K_5$ minus one edge). 

		Let $\vecT = (T_0, \ldots, T_{\ell})$ be such a variant of a triangle sequence with $T_{\ell} = B$ with the following properties, where $\rho = r(\vecT)$ is defined as in \eqref{eqn:rH}. 
		\begin{itemize}
			\item
				If possible, $T_0$ is a $K_5^-$. In this case we have $\rho = 0$, i.e.\ there are no irregular steps.
			\item
				Otherwise, if possible, $T_0$ is a $K_4$. In this case we have $\rho \le 1$, so there is at most one irregular step $i$, and then $v_{i}$ has exactly three neighbours in $T_{i-1}$.
			\item
				Otherwise, $B$ is $K_4$-free. If possible, avoid having irregular steps $i$ with $v_i$ having four neighbours in $T_{i-1}$. In this case $\rho \le 2$, so there are at most two irregular steps.
		\end{itemize}

		We now define, for each $i \in [0,\ell]$, an edge-colouring $\phi_{i}$ of $T_i$, using colours from $[0,i] \times [0,2]$ such that: 
		the colourings $\phi_i$ are nested (i.e.\ $\phi_{i+1}$ extends $\phi_i$ for every $i \in [0,\ell-1]$), and there are no rainbow triangles or monochromatic copies of $K_{1,3}$.

		\begin{itemize}
			\item
				Suppose that $i = 0$. 
				\begin{itemize}
					\item
						If $T_0$ is a triangle, colour two of its edges $(0,0)$ and the third $(0,1)$.
					\item
						If $T_0$ is a $K_4$, colour a $4$-cycle in it $(0,0)$, colour one of the two other edges $(0,1)$, and colour the final edge $(0,2)$. 
					\item
						If $T_0$ is a $K_5^-$, colour a $5$-cycle in it $(0,0)$ and colour all other edges $(0,1)$.
				\end{itemize}
			\item
				Now suppose $i \ge 1$.
				\begin{itemize}
					\item
						If $i$ is a regular step, colour the two edges from $v_i$ to $T_{i-1}$ by $(i,0)$ (and colour $T_{i-1}$ according to $\phi_{i-1}$).
					\item
						If $i$ is an irregular step with $\deg(v_i, T_{i-1}) = 3$, denote its three neighbours in $T_{i-1}$ by $u_1, u_2, u_3$, and proceed according to the following cases.
						\begin{itemize}
							\item
								If $\{u_1,u_2,u_3\}$ spans a single edge, then without loss of generality $u_3$ is isolated. Colour $v_i u_1$ and $v_i u_2$ by $(i,0)$ and colour $v_i u_3$ by $(i,1)$.
							\item
								If $\{u_1,u_2,u_3\}$ forms a triangle, then without loss of generality $u_1 u_2$ and $u_2 u_3$ are coloured by the same colour $c_1$ (as there are no rainbow triangles). 

								We claim that there are no other edges coloured $c_1$ in $\phi_{i-1}$.
								Indeed, the set $\{v_i, u_1, u_2, u_3\}$ forms a $K_4$, so $T_0$ is not a triangle.
                                Further, since $i$ is an irregular step, $T_0$ is not a copy of $K_5^-$. Thus $T_0$ is a copy of $K_4$, and this is the first and only irregular step.
								This shows, by the procedures for $T_0 = K_4$ and $i = 0$, and for regular steps, that all colour classes, except for $(0,0)$'s colour class, are paths of length at most $2$.
								If $c_1 = (0,0)$, then $u_1, u_2, u_3 \in V(T_0)$, but then $V(T_0) \cup \{v_i\}$ induces a $K_5^-$, a contradiction.
								Thus $c_1 \neq (0,0)$, so $c_1$'s colour class must be a path of length at most $2$, and indeed there are no edges coloured $c_1$ other than $u_1 u_2$ and $u_2 u_3$. 

								Colour $v_i u_1$ and $v_i u_3$ by $c_1$, and colour $v_i u_2$ by $(i, 0)$.
							\item
								Suppose that $u_1 u_2 u_3$ is a path coloured $c_1$, $u_1 u_3$ is a non-edge, and this is the first irregular step.

								We claim that there are no other edges coloured $c_1$.
								Indeed, $T_0$ is either a triangle or a $K_4$ (since $i$ is an irregular step), and if there is a colour with more than two edges then $T_0$ is a $K_4$, this colour is $(0,0)$, and vertices touching $(0,0)$ form a clique, so $c_1 \neq (0,0)$.

								Colour $v_i u_1$ by $c_1$, and colour $v_i u_2$ and $v_i u_3$ by $(i,0)$.
							\item
								Suppose that $u_1 u_2$ and $u_2 u_3$ are edges, coloured by distinct colours $c_1$ and $c_2$, $u_1 u_3$ is a non-edge, and this is the first irregular step.

								Without loss of generality, the edges coloured $c_1$ form a path of length at most $2$ (because so far we have at most one colour class which is not a path of length at most $2$). 
                                Thus for some $j \in [2]$ the vertex $u_j$ is not incident to other edges coloured $c_1$. 
								If $j = 1$ then colour $v_i u_1$ by $c_1$ and colour $v_i u_2$ and $v_i u_3$ by $(i,0)$.
								If instead $j = 2$ then colour $v_i u_1$ by $(i,0)$ and colour $v_i u_2$ and $v_i u_3$ by $c_1$. (See Figure~\ref{fig:4thsubcase}.)
                                \begin{figure}[!ht]
                                \begin{center}
                                \definecolor{qqwuqq}{rgb}{0,0.39215686274509803,0}
                                \definecolor{qqqqff}{rgb}{0,0,1}
                                \definecolor{ccqqqq}{rgb}{0.8,0,0}
                                \begin{tikzpicture}
                                \draw [line width=1pt,color=ccqqqq] (5,5)-- (3,3);
                                \draw [line width=1pt,color=qqqqff] (5,5)-- (7,3);
                                \draw [line width=1pt,color=qqwuqq] (5,5)-- (5,7);
                                \draw [line width=1pt,color=qqwuqq] (5,7)-- (7,3);
                                \draw [line width=1pt,color=ccqqqq] (5,7)-- (3,3);
                                \draw [line width=1pt,dashed=2pt,color=ccqqqq] (3,3)-- (5,2);
                                \draw [line width=1pt,color=ccqqqq] (13,7)-- (13,5);
                                \draw [line width=1pt,color=qqqqff] (13,5)-- (15,3);
                                \draw [line width=1pt,color=ccqqqq] (13,5)-- (11,3);
                                \draw [line width=1pt,color=qqwuqq] (11,3)-- (13,7);
                                \draw [line width=1pt,color=ccqqqq] (13,7)-- (15,3);
                                \draw [line width=1pt,dashed=2pt,color=ccqqqq] (13,5)-- (13,2);
                                \begin{scriptsize}
                                \draw [fill=black] (5,5) circle (1.5pt);
                                \draw [fill=black] (3,3) circle (1.5pt);
                                \draw [fill=black] (7,3) circle (1.5pt);
                                \draw [fill=black] (5,7) circle (1.5pt);
                                \draw [fill=black] (5,2) circle (1.5pt);
                                \draw [fill=black] (13,7) circle (1.5pt);
                                \draw [fill=black] (13,5) circle (1.5pt);
                                \draw [fill=black] (15,3) circle (1.5pt);
                                \draw [fill=black] (11,3) circle (1.5pt);
                                \draw [fill=black] (13,2) circle (1.5pt);
                                \end{scriptsize}
                                \draw (3,3) node[anchor=north] {\Large $u_1$};
                                \draw (5,5.2) node[anchor=west] {\Large $u_2$};
                                \draw (7,3) node[anchor=north] {\Large $u_3$};
                                \draw (5,7) node[anchor=south] {\Large $v_i$};
                                \draw (11.2,3) node[anchor=north] {\Large $u_1$};
                                \draw (13,5.2) node[anchor=west] {\Large $u_2$};
                                \draw (15,3) node[anchor=north] {\Large $u_3$};
                                \draw (13,7) node[anchor=south] {\Large $v_i$};
                                \end{tikzpicture}
								\caption{Colourings for $j=1$ (left) and $j=2$ (right) in the 4th subcase when $i\geq 1$ and $i$ is an irregular step with $\deg(v_i, T_{i-1}) = 3$. Here $\color{ccqqqq}{c_1 = \mathrm{red}}$, $\color{qqqqff}{c_2 = \mathrm{blue}}$ and $\color{qqwuqq}{(i,0) = \mathrm{green}}$.}\label{fig:4thsubcase}
                                \end{center}
                                \end{figure}
							\item
								Next, suppose that $u_1 u_2 u_3$ is a path coloured $c_1$, $u_1 u_3$ is a non-edge, and this is the second irregular step.

								We claim that for some $j \in \{1,3\}$, the vertex $u_j$ is not incident with other edges coloured $c_1$.
								Indeed, notice that in this case $T_0$ is a triangle and, from cases already considered, all colour classes are paths of length at most $4$. 
                                This means that if $u_1$ and $u_3$ are both incident to other edges coloured $c_1$ then there are vertices $u_0, u_4$ such that $u_0 \ldots u_4$ is a path of length $4$ coloured $c_1$, but then $u_1 u_3$ is an edge, a contradiction. Indeed, the only case where a path of length $4$ is formed is in the previous case, and then the second and fourth vertices are adjacent\footnote{These are the vertices $v_i$ and $u_1$ in the $j=2$ subcase of the previous case (see the right image of Figure~\ref{fig:4thsubcase}).}.
								
								Without loss of generality, $j = 1$. Colour $v_i u_1$ by $c_1$, and colour $v_i u_2$ and $v_i u_3$ by $(i,0)$.
							\item
								Finally, we may assume that $u_1 u_2$ and $u_2 u_3$ are edges coloured by distinct colours $c_1$ and $c_2$, $u_1 u_3$ is a non-edge, and this is the second irregular step.

								Again, we may assume that the edges coloured $c_1$ form a path of length at most $2$ (in this case $T_0$ is a triangle, so all colour classes, except for at most one\footnote{This would be the colour $c_1$ from either the 3rd or 4th subcase.}, are paths of length at most $2$). 

								As before, if $u_1$ is not incident with other edges coloured $c_1$, then colour $v_i u_1$ by $c_1$ and $v_i u_2$ and $v_i u_3$ by $(i,0)$.
								Otherwise, colour $v_i u_1$ by $(i,0)$ and colour $v_i u_2$ and $v_i u_3$ by $c_1$.
						\end{itemize}
					\item
						Now suppose $i$ is an irregular step with $\deg(v_i, T_{i-1}) = 4$. In particular, this is the first and only irregular step and $T_0$ is a triangle.
						Let $u_1, u_2, u_3, u_4$ be the neighbours of $v_i$, chosen so that $u_4$ is the last vertex to be added to the sequence.

						Then $\{u_1, u_2, u_3\}$ is an independent set, because otherwise we could have added $v_i$ to the sequence before $u_4$, thereby avoiding having irregular steps during which four edges are added. 
                        Note that $u_4$ has a neighbour in $\{u_1, u_2, u_3\}$, otherwise this is not a legal step of a triangle sequence. 
                        Assume without loss of generality $u_3u_4$ is an edge. 
                        Then $u_4u_1$ and $u_4u_2$ are non-edges. Indeed, suppose to the contrary that $u_j$ is a neighbour of $u_4$ for $j\in \{1,2\}$. 
                        This implies that $u_j u_3$ is an edge, because $u_4$ was added in a regular step, a contradiction.

						Colour $v_i u_1$ and $v_i u_2$ by $(i,0)$ and $v_i u_3$ and $v_i u_4$ by $(i,1)$.
						\qedhere
				\end{itemize}
		\end{itemize}
	\end{proof}

\subsection{Tightness when $k=3$ for \Cref{lem:k=3_triangle_colouring}}

	\Cref{lem:k=3_triangle_colouring} shows that if $G$ is a graph satisfying $m(G) < 2$ then $G \notconstrained (K_{1,3}, K_3)$.
	This bound is tight, as the following lemma shows, noting that $m(G) = 2 = m_2(K_3)$ for the graph $G$ shown in Figure~\ref{fig:tight_triangle}.

	\begin{lemma}\label{lem:tight_triangle}
		Let $G$ be a copy of $K_6$ with three edges forming a triangle removed. Then every edge--colouring of $G$ contains either a monochromatic $K_{1,3}$ or a rainbow triangle. 
	\end{lemma}
	\begin{proof}
		Label the vertices of $G$ with $1,\ldots,6$ such that $13$, $35$ and $15$ are non-edges, and all other pairs of vertices are edges, as shown in Figure \ref{fig:tight_triangle}.
		\begin{figure}[!ht]
        \centering
        \includegraphics{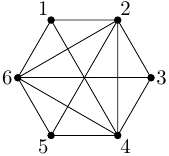}
        \caption{The graph $G$ in the statement of Lemma~\ref{lem:tight_triangle}}
        \label{fig:tight_triangle}
    \end{figure}
    
	Suppose for a contradiction that there exists a colouring of $G$ containing no monochromatic $K_{1,3}$ and no rainbow triangle. 
    Two of the edges of the triangle $246$ must be assigned the same colour. 
    By symmetry we can say without loss of generality that $46$ and $26$ are the same colour, say red. 

	As there is no monochromatic $K_{1,3}$, the edge $36$ must receive a different colour, say blue. 
    Triangles $346$ and $236$ are not rainbow so edges $23$ and $34$ must each be either red or blue. 
    They cannot both be blue else there would be a blue $K_{1,3}$ centred on vertex $3$, so at most one is blue. 
    By symmetry we can say without loss of generality that $23$ is red. 

	Consider now $25$ and $56$. 
    These are both adjacent to two red edges at $2$ and $6$ respectively, so cannot be red. 
    They must therefore be the same colour $c$, else triangle  $256$ is rainbow. Note that $c$ may be blue.

	Edge $45$ cannot be colour $c$, else there would be a monochromatic $K_{1,3}$ centred on vertex $5$. 
    Therefore $45$ must be red, else triangle $456$ would be rainbow. 

	Edge $24$ cannot be red, else there would be a red $K_{1,3}$ centred on vertex $2$ (or $4$). 
    Therefore $24$ must be colour $c$, else triangle $245$ would be rainbow. 

	Edge $34$ cannot be red, else there would be a red $K_{1,3}$ centred on vertex $4$. 
    Therefore $34$ must be colour $c$, else triangle $234$ would be rainbow. 

	Note that colour $c$ must be equal to blue, else  triangle $346$ is rainbow. 

	To conclude, observe that none of edges $12$, $14$ or $16$ may be red or blue without creating a monochromatic $K_{1,3}$.
    Triangle $124$ is not rainbow, so edges $12$ and $14$ are the same colour. 
    Moreover, triangle $146$ is not rainbow, so edges $14$ and $16$ are the same colour. 
    But then we have a monochromatic $K_{1,3}$ centred at $1$, a contradiction.
\end{proof}

\subsection{Proof of \Cref{thm:constrained,thm:constrained_K3}}

	We now combine the results from throughout this section to prove \Cref{thm:constrained,thm:constrained_K3}.

	\begin{proof}[Proof of \Cref{thm:constrained}]
		The 1-statement follows from Theorem~\ref{thm:1-statement}.

		First we consider the $0$-statement when $k \ge 4$, $m_2(H)=2$ and the only strictly $2$-balanced graph with 2-density $2$ contained in $H$ is $K_3$. Let $p \le n^{-1/2}/20$ (so $p \le cn^{-1/m_2(H)}$ for $c = 1/20$).
		We wish to show that, w.h.p., $G = G(n,p)$ can be edge-coloured so that there are no rainbow triangles or monochromatic $K_{1,4}$'s (this would then imply the statement for all $k \ge 4$, noting that if there are no rainbow triangles then there are no rainbow copies of $H$).
		In fact, it suffices to prove this for every triangle-connected component\footnote{Note that these `components' are not necessarily vertex disjoint but are edge disjoint.} of $G$ (namely a maximal triangle-connected subgraph of $G$), 
		since we can use disjoint sets of colours for different components and colour each edge not contained in any triangle with its own unique colour.
		By \Cref{lem:reduction-to-colouring-triangle} we may assume that every such component $H$ satisfies $e(H) \le 2v(H)$, and by \Cref{lem:k>3_triangle_colouring}, every such $H$ can be coloured appropriately.
		This proves the $0$-statement.

		For the $0$-statement when either $m_2(H) \ne 2$ or $m_2(H) =2$ and $H$ contains a strictly 2-balanced graph $J \ne K_3$ with $m_2(J) = 2$, note that it suffices to consider the case where $H$ is strictly 2-balanced. If not, take a minimal subgraph $H' \subseteq H$ which is not a triangle such that $m_2(H) = d_2(H')$, so that $H'$ is strictly 2-balanced. (Note that there exists such an $H'$ which is not a triangle by the conditions on $H$.)
        If a colouring of $G(n,p)$ contains no rainbow $H'$ then it also contains no rainbow $H$.
        Also note that it suffices to take $k = 3$ here, because avoiding a monochromatic $K_{1,3}$ also avoids a monochromatic $K_{1,k}$ for every $k \ge 3$. 
        Since $m_2(H) > 1$ is equivalent to $H$ not being a forest, \Cref{lem:notk3} immediately implies the required $0$-statement, noting that we may apply \Cref{lem:notk3} since $H$ is not a triangle.
	\end{proof}
 
For proving the final part of Theorem~\ref{thm:constrained_K3}, we will make use of the following fact. 
\begin{claim}\label{claim:subgraphappearance}
Let $F$ be a balanced graph. For any constant $c>0$ there exists $\zeta= \zeta(c) > 0$ such that if $p = cn^{-1/m(F)}$ then $\lim\limits_{n \to \infty}\mathbb{P}[F \subseteq G(n,p)] \ge \zeta$.
\end{claim}
While the proof of this claim is an elementary exercise in studying random graphs, we include its proof for completeness. We shall use Janson's inequality.
	\begin{theorem}[Janson's inequality {\cite[Theorem~8.1.1]{as}}]\label{thm:janson}
		Let $\Omega$ be a finite set, let $X$ be a random subset of $\Omega$ obtained by including each element $v$ in $\Omega$ with some probability $p_v$, independently.
		Let $\{A_i\}_{i \in I}$ be a collection of subsets of $\Omega$.
		Define $\mu$ and $\Delta$ as follows.
		\begin{align*}
			\mu & := \sum_{i \in I} \Pr\left(A_i \subseteq X\right), \\[.1em]
			\Delta & := \sum_{i \neq j \text{ and } A_i \cap A_j \neq \emptyset} \Pr\left(A_i \cup A_j \subseteq X\right).
		\end{align*}
		Then
		\begin{equation*}
			\Pr\left(A_i \not\subseteq X \text{ for every $i \in I$}\right) \le e^{-\mu + \Delta/2}.
		\end{equation*}
	\end{theorem}

\begin{proof}[Proof of Claim~\ref{claim:subgraphappearance}]
	We aim to apply Janson's inequality (\Cref{thm:janson}), with $\Omega$ as the set of unordered pairs of elements in $[n]$, $p_v := p = cn^{-1/m(F)}$, $X$ as a copy of $G(n,p)$, and $\{A_i\}_{i \in I}$ the collection of copies of $F$ in the complete graph on $[n]$. Note that we may assume $c$ is sufficiently small as subgraph inclusion is a monotone property, i.e.\ for any $c' > c$, $\Pr\left(F \subseteq G(n,c'n^{-1/m(F)})\right) \geq \Pr\left(F \subseteq G(n,cn^{-1/m(F)})\right).$
	Defining $\mu$ and $\Delta$ as in the theorem, we have
	\begin{equation*}
		\mu = \Theta(n^{v(F)} p^{e(F)}) = \Theta(c^{e(F)})
	\end{equation*} 
	since $F$ is balanced.
	Let $f(i)$ be the minimum number of vertices in a strict subgraph of $F$ with $i$ edges.
	By definition of $m(F)$, we have $f(i)\geq i/m(F)$.
	Then for every copy $F'$ of $F$ in $K_n$, there are at most $O(n^{v(F)-f(i)})$ copies $F''$ of $F$ where $|E(F') \cap E(F'')|=i$.
	Therefore, denoting by $\cF$ the collection of all copies of $F$ in $K_n$,
	\begin{align*}
		\Delta 
		& = \sum_{F' \in \cF} \sum_{\substack{F'' \in \cF \setminus \{F'\} : \\ E(F') \cap E(F'') \neq \emptyset}} \Pr\left(F' \cup F'' \subseteq G(n,p)\right) \\
		& = \sum_{F' \in \cF} \Pr\left(F' \subseteq G(n,p)\right) \sum_{\substack{F'' \in \cF \setminus \{F'\} : \\ E(F') \cap E(F'') \neq \emptyset}} \Pr\left(F'' \setminus E(F') \subseteq G(n,p)\right) \\
		& \le \mu \sum_{i \in [e(F)-1]} O(n^{v(F)-f(i)} p^{e(F)-i}) \\
		& \le \mu \sum_{i \in [e(F)-1]} O(n^{v(F)-i/m(F)} p^{e(F)-i}) = O(\mu c) < \mu,
	\end{align*} 
	since $c$ is sufficiently small.
	
	 Applying \Cref{thm:janson}, we get
	\begin{equation*}
		\Pr\left(F \subseteq G(n,p)\right)
		\ge 1-e^{-\mu + \Delta/2}
		= 1-e^{-\mu/2}
		= 1 - e^{-\Theta(c^{e(F)})}. \qedhere
	\end{equation*}
\end{proof}

\begin{proof}[Proof of \Cref{thm:constrained_K3}]
	The 1-statement follows from Theorem~\ref{thm:1-statement}.

	For the $0$-statement, let $p = o(n^{-1/2})$ (notice that $m_2(H) = 2$ so this is the same as requiring $p = o(n^{-1/m_2(H)})$).
	We would like to show that, w.h.p., $G = G(n,p)$ can be edge-coloured without rainbow triangles and monochromatic copies of $K_{1,3}$. (Again, if there are no rainbow triangles then there are no rainbow copies of $H$.)
	By \Cref{lem:reduction-to-colouring-triangle}, we may assume that every triangle-connected subgraph $H$ of $G$ satisfies $e(H) < 2v(H)$.
	By using disjoint colour sets for different triangle-connected components, it suffices to show that every triangle-connected component can be coloured without rainbow triangles or monochromatic $K_{1,3}$'s, and this is indeed the case by \Cref{lem:k=3_triangle_colouring}.
	This completes the proof of the $0$-statement.

	For the final part of the theorem, we apply Claim~\ref{claim:subgraphappearance} with $F$ the graph defined in Lemma~\ref{lem:tight_triangle}, noting that $m(F)=m_2(K_3)=2$, $F \constrained (K_{1,3},K_3)$ and one can check that $F$ is balanced. 
 \qedhere
\end{proof}

\section{Concluding remarks}\label{sec:conclusion}

In this paper, we closed the gap on locating the threshold for the constrained Ramsey property for $(H_1,H_2)$ for all cases except for when $H_1=K_{1,2}$. 
For graphs $H_1$ and $H_2$, let $m_{\cram}(H_1,H_2):= \inf \{ m(G) : G \constrained (H_1,H_2) \}$. 
We present the full results in Table~\ref{tab:constrained}.
Note that in~\cite{c-ram_threshold}, all threshold functions obtained are coarse thresholds. 
In the table, we mark all locations where the coarse threshold could be improved to a semi-sharp threshold.
Note that we only consider graph $H_2$ with at least two edges, because if $H_2$ has one edge (or no edges) then every copy of $H_2$ is rainbow.

\begin{table}[h!]
\renewcommand{\arraystretch}{1.8}
\setstretch{.8}
\renewcommand{\footnotelayout}{\setstretch{1}}
\setlength{\footnotesep}{1.1\baselineskip}
\begin{minipage}{\textwidth} 
\begin{tabularx}{\textwidth}{@{}XXlllll@{}} \toprule
& & & \multicolumn{2}{c}{$1$-statement} & \multicolumn{2}{c}{$0$-statement}\\
\cmidrule(r){4-5} \cmidrule{6-7} 
 $H_1, \, e = e(H_1)$ & $H_2, \, e = e(H_2)$ & Threshold $n^{-1/f(H_1,H_2)}$ & Ref & Type & Ref & Type \\
 \midrule
  $K_{1,k}$, $k \geq 3$ & Not forest or in $\mathcal{H}^*$ \footnote{We write $\mathcal{H}^{*}$ for the class of graphs $H$ which satisfy $m_2(H)=2$ and the unique strictly $2$-balanced graph $J$ contained in $H$ with $m_2(J)=2$ is $K_3$.} & $m_2(H_2)$ & \ref{thm:1-statement}\cite{anti-ram_1-statement} & $\ge C$ & \ref{thm:constrained} & $\le c$ \\ 
 $K_{1,k}$, $k \geq 4$ & In $\mathcal{H}^*$ & $m_2(H_2) = 2$ &  \ref{thm:1-statement}\cite{anti-ram_1-statement} & $\ge C$ & \ref{thm:constrained} & $\le c$ \\ 
 $K_{1,3}$ & $K_3$ & $m_2(H_2) 
 = 2$ & \ref{thm:1-statement}\cite{anti-ram_1-statement} & $\ge C$ & \ref{thm:constrained_K3} & $\ll$ \\ 
  $K_{1,3}$ & In $\mathcal{H}^*$ but not $K_3$ & $m_2(H_2) 
 = 2$ & \ref{thm:1-statement}\cite{anti-ram_1-statement} & $\ge C$ & \ref{thm:constrained_K3} & $\ll$\footnote{Coarse ($\gg / \ll)$ could potentially 
 be improved to semi-sharp ($ \ge C / \le c$) here.\label{fn:coarseish}} \\ 
 \midrule
 Not star forest & Forest, ${e \geq 3}$ & $m_2(H_1)$ & \cite{c-ram_threshold} & $\gg$\footref{fn:coarseish} & \cite{rrlower}\footnote{Follows from the random Ramsey threshold.} & $\le c$\footnote{With the exception of when $H_1$ is a path of $3$ edges, since in this case the random Ramsey threshold for the $0$-statement is only coarse.} \\ 
 Star~forest, not star & Forest, not~short\footnote{A \emph{short} forest is defined to be one where all components have at most two edges, that is, a disjoint union of $K_2$s and $P_3$s.\label{fn:short}} & $m_2(H_2)= 1$ & \cite{c-ram_threshold} & $\gg$\footref{fn:coarseish} & \cite{c-ram_threshold} & $\ll$\footref{fn:coarseish}\footnote{If $H_1$ is a matching of size $2$, $H_2$ is a path of $3$ edges, $G$ a triangle with a path of length $2$ added to each vertex of the triangle, then we have $m(G)=1$ but $G \constrained (H_1,H_2)$ so for this example, the $0$-statement is coarse. Note that one could also take $G = C_5$ here to achieve the same conclusion.}\\ 
 Star~forest, not star & Short forest\footref{fn:short} & $m_{\cram}(H_1,H_2)$ & \cite{c-ram_threshold} & $\gg$\footref{fn:coarseish} & \cite{c-ram_threshold} & $\ll$ \\ 
   $K_{1,k}$, $k \geq 2$ & Forest, $e \geq 3$ & $m_{\cram}(H_1,H_2)$ & \cite{c-ram_threshold} & $\gg$\footref{fn:coarseish} & \cite{c-ram_threshold} & $\ll$ \\ 
 Any, $e \geq 2$ & $K_2 \sqcup K_2$ & $m_{c-ram}(H_1,H_2) = m(H_1)$ & \cite{c-ram_threshold} & $\gg$\footref{fn:coarseish} & \cite{c-ram_threshold} & $\ll$ \\ 
 Not forest, ${e \geq 2}$ & $P_3$ (`cherry') & $m_{\cram}(H_1,H_2) = m(H_1)$ & \cite{c-ram_threshold} & $\gg$\footref{fn:coarseish} & \cite{c-ram_threshold} & $\ll$ \\ 
 Forest, {$k$ non-isolated vertices} & $P_3$ (`cherry') & $m_{\cram}(H_1,H_2) = \frac{k-1}{k}$ & \cite{c-ram_threshold} & $\gg$\footref{fn:coarseish} & \cite{c-ram_threshold} & $\ll$ \\ 
\bottomrule
\end{tabularx}
\vspace{-.4cm}
\caption{Thresholds for constrained Ramsey}
\vspace{.2cm}
\label{tab:constrained}
\end{minipage}
\end{table}
For a graph property $\cP$, a function $f: \mathbb{N} \to \mathbb{R}$ is called a \emph{sharp threshold function}, if for all $\eps>0$, 
\[\lim\limits_{n \to \infty} \mathbb{P}[G(n,p) \in \cP] =   
\begin{cases}
    0       & \quad \text{if } p \leq (1-\eps) f(n),\\
    1       & \quad \text{if } p \geq (1+\eps) f(n).
\end{cases}\] 
We are not aware of any results on sharp thresholds for the constrained Ramsey property, 
and it would be interesting to determine which semi-sharp thresholds can be improved to sharp. 
Sharp thresholds occur for some cases of the Ramsey property for graphs.
Very recently, Friedgut, Kuperwasser, Samotij and Schacht~\cite{fkss} obtained a sharp threshold for the Ramsey property for $H$,
whenever $H$ is strictly $2$-balanced, not a forest, and so-called \emph{collapsible}.
Previously Friedgut and Krivelevich~\cite{fk} found a characterisation of when there is a sharp threshold for the Ramsey property for $H$ when $H$ is a tree.

 For the constrained Ramsey problem for $(K_{1,2},H_2)$, which as earlier discussed is equivalent to having the anti-Ramsey property for $H_2$,
see Table~\ref{tab:anti} for all known results. 
Note that we write $m_{\aram}(H_2):=m_{\cram}(K_{1,2},H_2)$.

\begin{table}[h!]
\renewcommand{\arraystretch}{1.8}
\setstretch{.8}
\renewcommand{\footnotelayout}{\setstretch{1}}
\setlength{\footnotesep}{1.1\baselineskip}
\begin{minipage}{\textwidth} \centering
\begin{tabularx}{\textwidth}{@{}>{\raggedright\arraybackslash}Xlllll@{}} \toprule
 & & \multicolumn{2}{c}{$1$-statement} & \multicolumn{2}{c}{$0$-statement}\\
\cmidrule(r){3-4} \cmidrule{5-6} 
  $H$ & Threshold $n^{-1/f(H)}$ & Ref & Type & Ref & Type \\
 \midrule 
 Any & $\le m_2(H)$ &  \ref{thm:1-statement}\cite{anti-ram_1-statement} & $\ge C$ & & \\ \midrule
 $K_k$, $k \ge 5$  & $m_2(H)=(k+1)/2$ & \ref{thm:1-statement}\cite{anti-ram_1-statement} & $\ge C$ & \cite{npss,anti-ram_complete} & $\ll$\footnote{Coarse ($\gg / \ll)$ could potentially 
 be improved to semi-sharp ($ \ge C / \le c$) here.\label{fn:coarseish2}}\footnote{$\le c$ proved for $k \geq 19$ in~\cite{npss}.}\\
  $C_k$, $k \ge 5$ & $m_2(H)=(k-1)/(k-2)$ & \ref{thm:1-statement}\cite{anti-ram_1-statement} & $\ge C$ & \cite{npss,anti-ram_cycles} & $\ll$\footref{fn:coarseish2}\footnote{$\le c$ proved for $k \geq 7$ in~\cite{npss}.}\\
   $K_3$  & $m_{\aram}(H)=1$  & \cite{bo} & $\gg$ & \cite{bo} & $\ll$\\
 $K_4$  & $m_{\aram}(H)=15/7$ & \cite{anti-ram_complete} & $\gg$\footref{fn:coarseish2}  & \cite{anti-ram_complete} & $\ll$ \\
  $K_4$ minus edge  & $m_{\aram}(H)=3/2$ & \cite{anti-ram_complete} & $\gg$\footref{fn:coarseish2}  & \cite{anti-ram_complete} & $\ll$ \\
  $C_4$ & $m_{\aram}(H)=4/3$ &  \cite{anti-ram_cycles} & $\gg$\footref{fn:coarseish2} & \cite{anti-ram_cycles} & $\ll$ \\
  Forest & $m_{\aram}(H)$ & \cite{c-ram_threshold} & $\gg$\footref{fn:coarseish2} & \cite{c-ram_threshold} & $\ll$ \\
   $B_t \oplus F$, $B_t$ book, ${1 < m_2(F) < 2}$ & $<2 \le m_2(H)$ & \cite{anti-ram_non-balanced}\footnote{The case $t=1$ was proved earlier in~\cite{kkm}.} & $\ge C$ & &\\
Strictly 2-balanced and ${1 < m_2(H) < \frac{\delta(H)(\delta(H) + 1)}{2\delta(H) + 1}}$ & $m_2(H)$ & \ref{thm:1-statement}\cite{anti-ram_1-statement} & $\ge C$ & \ref{thm:antiramsey_specialcase} & $\le c$ \\ 
$H$ with $m_2(H) \geq 19$ & $m_2(H)$ & \ref{thm:1-statement}\cite{anti-ram_1-statement} & $\ge C$ & \cite{kuperwasser} & $\le c$ \\
 \bottomrule
 \end{tabularx}
  \caption{Thresholds for anti-Ramsey 
  }
 \label{tab:anti}
  \end{minipage}
 \end{table}

Implicitly within our work in Section~\ref{sec:constrained}, we prove the following analogue of Theorem~\ref{thm:antiramsey0} for the constrained Ramsey property.

\begin{theorem}\label{thm:constrained0}
	Let $H$ be a strictly $2$-balanced graph with $m_2(H)> 1$ and let $k \ge 3$. To prove a semi-sharp $0$-statement for the constrained Ramsey property for $(K_{1,k},H)$ at $n^{-1/m_2(H)}$, it suffices to prove for all graphs $G$ with $m(G) \leq m_2(H)$ that $G \notconstrained (K_{1,k},H)$.  
	The same statement holds with `semi-sharp' replaced by `coarse' and `$m(G) \leq m_2(H)$' replaced by `$m(G) < m_2(H)$'.
\end{theorem}

Observe by Claim~\ref{claim:subgraphappearance}, if there exists a balanced graph $G$ with $m(G)=m_2(H)$ such that $G \constrained (K_{1,k},H)$, then a coarse threshold at $n^{-1/m_2(H)}$ is the best we can hope for. This is 
exactly the case for the constrained Ramsey property for $(K_{1,3},K_3)$, as seen in Theorem~\ref{thm:constrained_K3}. 

Further note that analogous results to Theorem~\ref{thm:constrained0} hold for the Ramsey property, see, e.g.~\cite{ns}, and the asymmetric Ramsey property (when one looks for a different graph in each different colour), see~\cite{bhh,ksw}.
That is, roughly speaking, proving a $0$-statement at the natural candidate location can be reduced to solving an appropriate colouring question for not too dense graphs.

For graphs $H$ where the colouring statement within Theorem~\ref{thm:antiramsey0} does not hold, it is tempting to speculate that the threshold is at $n^{-1/m_{\aram}(H)}$, as is the case when $H$ is a copy of $K_3, K_4, C_4$ or any forest. 
An analogous result is true for the Ramsey property for graphs. Indeed, the only case for which the threshold is not $n^{-1/m_2(H)}$ is the case of stars $K_{1,k}$ which have a threshold given by the appearance of $K_{1,r(k-1)+1}$ (which whenever $r$-coloured gives a monochromatic copy of $K_{1,k}$).
However, in the case of the Ramsey property for hypergraphs, it was shown in~\cite{gnpsst} that for each $k \geq 4$, there exists a $k$-uniform hypergraph $H$ such that the threshold is not given by either the natural generalisation of $n^{-1/m_2(H)}$ nor the threshold for the appearance of a hypergraph which has the Ramsey property for $H$.
See~\cite{gnpsst} or e.g.~\cite{bhh} for further discussion on the thresholds of the Ramsey property for hypergraphs.

For a graph $H$, finding a particular graph $G$ with $m(G) \leq m_2(H)$ for which $G \antiramsey H$ immediately gives rise to a coarse $1$-statement for the anti-Ramsey property for $H$ at $n^{-1/m(G)}$.
However, this is not how the $1$-statement for $B_t \oplus F$ was proved in~\cite{anti-ram_non-balanced}. 
They instead used the appearance of lots of copies of $B_{3t-2}$ in $G(n,p)$ at the stated threshold and the property that $B_{3t-2} \antiramsey B_t$, to extend some rainbow copy of $B_t$ to a rainbow copy of $B_t \oplus F$. 

The colouring statement within Theorem~\ref{thm:antiramsey0} seems very hard to prove or disprove in general, even for specific $H$; 
indeed, known results about this statement and the analogous statement for other Ramsey properties make use of structural properties of $H$.  
It would be of great interest to determine the location of the threshold for the remaining open cases of the anti-Ramsey property, in particular by obtaining results on the colouring statement within Theorem~\ref{thm:antiramsey0}. 

{\bfseries Additional remark.} Since this paper was first submitted, Kuperwasser~\cite{kuperwasser} has made some progress on this exact question, proving that the answer to Question~\ref{ques:antiramsey} is yes whenever $H$ has $m_2(H) \geq 19$. 

\section*{Acknowledgements}
The authors would like to thank the two anonymous referees for their careful and helpful reviews, and
in particular, for a simplification of the proof of Lemma~\ref{lem:c4}, and for observations which led to a correction to Theorems~\ref{thm:constrained} and~\ref{thm:constrained_K3}.

\bibliography{Tash_version}
\bibliographystyle{amsplain}
\end{document}